\documentclass{article}
\usepackage{fullpage}
\usepackage[affil-it]{authblk}
\usepackage[numbers,sort]{natbib}
\usepackage{amsthm,amssymb}

\theoremstyle{plain}
\newtheorem{theorem}{Theorem}
\newtheorem{lemma}{Lemma}
\newtheorem{corollary}{Corollary}

\usepackage{bm}
\usepackage{multirow} % for table
\usepackage{booktabs} % for table
\usepackage{tikz}
\usepackage{tikz-network}
\usepackage{url}
\usepackage{subcaption}

\usepackage{mathtools}

\usepackage{algpseudocode}
\usepackage{algorithm}

\usepackage{color}

\usepackage{hyperref}
\begin{document}

\title{Complex non-backtracking matrix for directed graphs}

\author[1]{Keishi Sando%
	\thanks{Corresponding author: \texttt{sando.keishi.sp@alumni.tsukuba.ac.jp}}}
\affil[1]{Department of Statistical Science, The Graduate University for Advanced Studies, Kanagawa, JP}

\author[2]{Hideitsu Hino}
\affil[2]{The Institute of Statistical Mathematics, 10-3 Midoricho, Tokyo, JP}

\date{}

\maketitle

\begin{abstract}
	{Graph representation matrices are essential tools in graph data analysis. Recently, Hermitian adjacency matrices have been proposed to investigate directed graph structures. Previous studies have demonstrated that these matrices can extract valuable information for clustering. In this paper, we propose the complex non-backtracking matrix that integrates the properties of the Hermitian adjacency matrix and the non-backtracking matrix. The proposed matrix has similar properties with the non-backtracking matrix of undirected graphs. We reveal relationships between the complex non-backtracking matrix and the Hermitian adjacency matrix. Also, we provide intriguing insights that this matrix representation holds cluster information, particularly for sparse directed graphs.
	} % Abstract
	{directed graphs; Hermitian adjacency matrix; non-backtracking matrix; spectral clustering.} % Keywords
	%%%% If classification number provided then
	\\
	% 2000 Math Subject Classification: 34K30, 35K57, 35Q80,  92D25
\end{abstract}

\section{Introduction}

In network analysis, various matrix representations have been developed to investigate the structural properties of the corresponding network,
among which the non-backtracking (NBT) matrix is one such representation.
The NBT matrix is well-known for its relationship with the Ihara zeta function~\citep{Ihara1966-bk} defined as an infinite product over equivalence classes of primitive cycles.
It was shown in~\citep{Ihara1966-bk,Sunada1986-kg} that, for regular graphs,
the reciprocal of the Ihara zeta function can be expressed as a polynomial related to the adjacency matrix.
The relation between the zeta function and the determinant of the NBT matrix was elucidated in~\citep{Hashimoto1989-su},
extended to irregular graphs in~\citep{Bass1992-ek}, and an elementary proof was provided in~\citep{Stark1996-mx}.
The connection between the polynomial and the determinant of the NBT matrix, via the Ihara zeta function, is known as the Ihara's formula.
Not limited to undirected graphs, variants of Ihara's formula have been studied in various contexts such as directed~\citep{Tarfulea2009-sy}, weighted~\citep{Kempton2016-ev,Konno2019-jd} settings.
The NBT matrix and matrices derived from the NBT matrix can reveal graph structures through their spectra~\citep{Jost2023-pj,Mulas2024-en}
and are also utilized for graph clustering problem~\citep{Krzakala2013-up,Newman2013-zc,Saade2014-vg,Bordenave2018-qo,Gulikers2017-iu}, node centrality computation~\citep{Martin2014-tv,Grindrod2018-va,Torres2021-nr,Timar2021-es} and graph isomorphism problem~\citep{Mulas2024-en}.
Methods developed for these problems have also found applications in other fields, such as the analysis of epidemic dynamics~\citep{Shrestha2015-ih, Radicchi2016-rm}.

To investigate graph structures, research on the spectra of undirected graphs has been conducted for a long time.
Several characteristics of undirected graphs, including the number of connected components, bipartite structure, and diameter,
have well-established relationships with eigenvalues of matrices~\citep{Chung1996-hp}.
In contrast, the study of eigenvalues in directed graphs remains relatively underexplored.
One of the reasons is that the adjacency matrix of directed graphs, being asymmetric, typically yields complex eigenvalues, complicating the analysis.
To address this challenge while preserving the intrinsic information of directed graphs, the Hermitian adjacency matrices have been introduced in~\citep{Guo2017-or,Liu2015-qu}.
Basic properties of the Hermitian adjacency matrices are discussed in~\citep{Guo2017-or},
and Hermitian energies, defined as the absolute sum of eigenvalues, were studied in~\citep{Liu2015-qu}.

Hermitian adjacency matrices have been used in clustering problems, especially for spectral clustering.
For undirected graphs, the traditional objective is to identify clusters
that are densely connected internally and sparsely connected between clusters.
Although these methodologies can be applied to directed graphs by symmetrization~\citep{Satuluri2011-ja},
these may fail to capture valuable structures derived from directed information.
To address this, spectral clustering methods for directed graphs based on a Hermitian adjacency matrix have recently been proposed.
In~\citep{Cucuringu2020-de}, their proposed algorithm was evaluated using the directed stochastic block model.
The directional relationships between clusters were formulated as a graph cut problem in \citep{Laenen2020-ft}.
In~\citep{Hayashi2022-sy}, they derived an efficient algorithm by utilizing the property that the employed Hermitian matrix is both Hermitian and skew-symmetric, and evaluated the flow imbalance in directed graphs.
In contrast to the previously mentioned Hermitian adjacency matrices, where complex values are determined by the edge connections between vertex pairs, a variant of the Hermitian adjacency matrix that incorporates the cluster structure into its definition is proposed in ~\citep{Martin2024-en}, along with an iterative algorithm designed for cluster estimation using this variant.

In certain conditions involving undirected graphs, it has been observed that methods derived from the adjacency matrix yields unreliable outcomes.
For instance, in node clustering on sparse graphs, these methods often result in misidentified clusters.
Moreover, in node centrality, the presence of hub vertices causes the importance to concentrate on a few vertices, making it difficult to distinguish between the remaining vertices.
It has been reported that the NBT walks and the NBT matrix can mitigate these issues~\citep{Krzakala2013-up,Newman2013-zc,Martin2014-tv,Arrigo2018-vi},
whereas few studies have aimed at elucidating its properties for both undirected and directed graphs,
and its relationship with the Hermitian adjacency matrix has not yet been established.

The issue arising from methods derived from the adjacency matrix under specific conditions, as previously mentioned, is likely to also occur with the Hermitian adjacency matrix for directed graphs as well.
We demonstrate this in Section~\ref{section:application} through a numerical experiment.
In this work, as a step toward developing a method that can be robustly applied across various situations,
we introduce a complex non-backtracking~(CNBT) matrix for directed graphs that incorporates characteristics of both the Hermitian adjacency matrix and the NBT matrix.
Our matrix representation exhibits favorable properties, analogous to the relationship between the adjacency matrix and the NBT matrix in undirected graphs.
Additionally, we show through experiments that in the node clustering problem, spectral clustering based on our matrix performs effectively in scenarios where methods based on the Hermitian adjacency matrix lead to misclassification.

The rest of this paper is organized as follows.
In Section~\ref{section:preliminaries}, we fix notations and give basic properties on NBT walks.
Section~\ref{section:relation_with_herm_adj_mat} is dedicated to the main results, which corroborate the validity of the proposed non-backtracking matrix.
Additionally, in Section~\ref{section:application}, we demonstrate that our clustering algorithm outperforms an existing method when the underlying graph is sparse.
This result further supports the usefulness of our definition.
In Section~\ref{section:conclusion}, we give a conclusion on this paper.

\section{Preliminaries}\label{section:preliminaries}

\subsection{Notations}

Throughout this paper, let $G=(V,E)$ denotes a directed simple graph without self-loops.
We assume that the vertices of $G$ are indexed, and we use each vertex and its corresponding index interchangeably.
For an arbitrary directed edge $e \in E$, let $i_{e}, t_{e} \in V$ be the initial and terminal vertex of $e$, respectively.
In this paper, the symbol $i$ is used in the following ways: $i_{e}$ denotes the initial vertex of a directed edge $e$.
When used alone or as a subscript, $i$ serves as an index. The imaginary unit is represented as $\sqrt{-1}$.

When a directed edge $e$ is oriented from $u\in V$ to $v \in V$, we use $e=\overrightarrow{uv}$.
Its inverse, the directed edge from $v$ to $u$, is denoted by $e^{-1} = \overrightarrow{vu}$.
An undirected edge between $u$ and $v$ is represented as $uv~(=vu)$.
For any digraph $G$, we write $U_G = (V, E_U)$ to denote the underlying undirected graph of $G$, which is the undirected graph with the same vertex set of $G$, and whose edge set $E_U$ is defined by replacing each directed edge $\overrightarrow{uv} \in E$ with the undirected edge $uv \in E_U$. If both $\overrightarrow{uv}$ and $\overrightarrow{vu}$ are included in $E$, these two directed edges give rise to the same undirected edge $uv \in E_U$.
We can write $E_U = \left\{ uv \mid u,v \in V,\; \overrightarrow{uv} \in E \;\text{or}\; \overrightarrow{vu} \in E \right\}$.
An undirected edge $uv \in E_U$ is referred to as an unoriented edge.
Let $n\coloneqq \left| V \right|$ be the number of vertices in $G$, and $m\coloneqq \left| E_U \right|$ be the number of unoriented edges in $U_G$.

We introduce $\bar{E}$ and $\vec{E}$ to refer to directed edges with assigned indices.
Let $\bar{E} \coloneqq \left\{ \overrightarrow{uv}, \overrightarrow{vu} \mid uv \in E_U \right\}$ be the set of directed edges obtained by orienting each edge of $U_G$ in both directions.
We then define $\vec{E}$ as an indexed set of $\bar{E}$ under the constraint that $e_{m+i} = e_{i}^{-1},\; i=1,\dots,m$.
Note that $E \subset \bar{E}$ and $| \bar{E} | = | \vec{E} | = 2m$ hold.
Unless ambiguity arises, we will not distinguish between $e \in \vec{E}$ and its assigned index.

To indicate the presence of directed edges between two vertices in $G$,
we introduce the notation $u \rightarrow_{G} v$ to denote $\overrightarrow{uv} \in E \;\text{and}\; \overrightarrow{vu} \not\in E$~\citep{Liu2015-qu}.
We treat the notations $u \rightarrow_{G} v$ and $v \leftarrow_{G} u$ as equivalent and interchangeable.
Similarly, we write $u\leftrightarrow_{G} v$ to denote $\overrightarrow{uv}, \overrightarrow{vu} \in E$.
In appendix~\ref{appendix:notations}, we write down a concrete example.

For any vertex $u \in V$, we define four types of neighborhoods as follows:
\begin{align*}
	\overleftarrow{N_u}      & \coloneqq \left\{ v \in V \mid u \leftarrow_{G} v \right\}
	,                        & \quad
	\overrightarrow{N_u}     & \coloneqq \left\{ v \in V \mid u\rightarrow_{G} v \right\}
	,                                                                                                                \\
	\overleftrightarrow{N_u} & \coloneqq \left\{ v \in V \mid u\leftrightarrow_{G} v \right\}
	,                        & \quad
	N_u                      & \coloneqq \overleftrightarrow{N_u} \cup \overleftarrow{N_u} \cup \overrightarrow{N_u}
	.
\end{align*}
We use $\overleftarrow{d_u} \coloneqq | \overleftarrow{N_u} |,\; \overrightarrow{d_u} \coloneqq | \overrightarrow{N_u} |,\; \overleftrightarrow{d_u} \coloneqq | \overleftrightarrow{N_u} |,\; d_u \coloneqq \overleftarrow{d_u} + \overrightarrow{d_u} + \overleftrightarrow{d_u}$ and $D \coloneqq \text{diag}(d_1,\dots, d_n) \in \mathbb{R}^{n\times n}$ to represent the in-degree, out-degree, bi-degree and degree of $u$.

A sequence of vertices $W = (v_1,\dots, v_k)$ of a digraph $G$ is often called a walk if,
for any two consecutive vertices $v_i$ and $v_{i+1}$ in the sequence,
there exists a directed edge from $v_i$ to $v_{i+1}$ in $E$.
We define the length of a walk $W$ as the number of edges traversed in the walk, and we denote this length by $|W|$.
A sequence of vertices is called a mixed walk of $G$~\citep{Liu2015-qu} when there exists an unoriented edge between any two consecutive vertices in $E_U$.
Namely, a mixed walk is allowed to traverse edges against their orientation.

A mixed walk $W=(v_1, \dots, v_k)$ has backtracking if there exists $i\in \{ 2,\dots, k-1 \}$ such that $v_{i-1} = v_{i+1}$.
When a mixed walk does not have backtracking, it is called a NBT mixed walk.
A mixed walk is closed if its initial and terminal vertices coincide; such a mixed walk is called a cycle.
The $r$ times repetition of a cycle $C$ is denoted as $C^r$.
Even if a cycle $C$ is NBT, $C^2$ is not necessarily NBT.
A cycle $C$ is said to have a tail if $C$ is NBT but $C^2$ is not~\citep{Stark1996-mx}.
A cycle is primitive when it is not a repetition of another cycle.
Two cycles are considered equivalent if they follow the same sequence of vertices but start from different vertices.
Let $[C]$ be the equivalence class which includes a cycle $C$.

We propose to introduce the concept of ``rotation'' in a mixed walk.
Let $R \in \mathbb{N}$ be the maximum rotation number we can arbitrarily determine.
For $\overrightarrow{uv} \in \vec{E}$, we define its rotation $r(u,v)$ as follows:
\begin{align*}
	r(u,v) \coloneqq \begin{cases}
		                 1  & u \rightarrow_{G} v
		                 ,                            \\
		                 -1 & u\leftarrow_{G} v
		                 ,                            \\
		                 0  & u \leftrightarrow_{G} v
		                 .
	                 \end{cases}
\end{align*}
This value $r(u,v)$ indicates whether traversing from $u$ to $v$ aligns with the direction of the corresponding edge in $G$. Note that traversals between two vertices connected by bidirectional edges are not counted.
The rotation of a mixed walk $W=(v_1,\dots, v_k)$ is defined by $r(W) \coloneqq \sum_{l=1}^{k-1} r(v_l, v_{l+1}) \text{ mod } R$.
This quantity represents a remainder of the number of times directed edges of mixed walk $W$ are traversed in their correct orientation in $G$.

In this paper, let $\alpha \in \mathbb{C}$ be a complex number with an absolute value of $1$,
and we set the maximum rotation number $R$ as the smallest positive integer satisfying $\alpha^R = 1$.
Although several definitions of a Hermitian adjacency matrix have been proposed~\citep{Guo2017-or,Liu2015-qu,Mohar2020-fa,Martin2024-en},
we focus on the Hermitian adjacency matrix $A_{\alpha} \in \mathbb{C}^{n\times n}$ of a directed graph $G$ defined as follows~\citep{Guo2017-or,Liu2015-qu}:
\begin{align}
	\left( A_{\alpha} \right)_{uv} \coloneqq \begin{cases}
		                                         1            & u \leftrightarrow_{G} v
		                                         ,                                      \\
		                                         \alpha       & u\rightarrow_{G} v
		                                         ,                                      \\
		                                         \bar{\alpha} & u\leftarrow_{G} v
		                                         ,                                      \\
		                                         0            & \text{otherwise}
	                                         \end{cases}
	\quad u,v \in V
	,\label{definition:A_alpha}
\end{align}
where $\bar{\alpha}$ is a conjugate complex number of $\alpha$.
The matrix $A_{\alpha}$ represents the adjacency relationship between vertices.
We use $A \in \mathbb{R}^{n\times n}$ as the standard adjacency matrix.
Next, we also introduce a matrix representation that describes the adjacency relationship between directed edges.
Let $\delta_{uv}$ be the Kronecker delta function and
$B \in \mathbb{R}^{2m\times 2m}$ be the standard NBT matrix,
whose $(e,f)$ element is given by $(B)_{ef} \coloneqq \delta_{t_e i_f} (1 - \delta_{i_e t_f})$.
We define a directed edge type $\lambda_e \in \mathbb{C}$ of $e \in \vec{E}$, and propose the complex non-backtracking matrix $B_{\alpha} \in \mathbb{C}^{2m\times 2m}$ as follows:
\begin{align}
	B_{\alpha} \coloneqq B \Lambda
	,\quad \text{where}\quad
	\lambda_e  \coloneqq \begin{cases}
		                     1            & i_e \leftrightarrow_{G} t_e
		                     ,                                          \\
		                     \alpha       & i_e \rightarrow_{G} t_e
		                     ,                                          \\
		                     \bar{\alpha} & i_e \leftarrow_{G} t_e
		                     ,
	                     \end{cases}
	\quad
	\Lambda    \coloneqq \text{diag}\begin{pmatrix}
		                                \lambda_1 & \dots & \lambda_{2m}
	                                \end{pmatrix}
	.\label{definition:B_alpha}
\end{align}

Let $\overleftrightarrow{A}, \overrightarrow{A}, \overleftarrow{A} \in \mathbb{R}^{N\times N}$ be
the matrices where $( \overleftrightarrow{A} )_{uv} = 1$ if $u\leftrightarrow_{G} v$ and $0$ otherwise,
$( \overrightarrow{A} )_{uv} = 1$ if $u\rightarrow_{G} v$ and $0$ otherwise, and
$( \overleftarrow{A} )_{uv} = 1$ if $u\leftarrow_{G} v$ and $0$ otherwise, respectively.
Similarly, define $\overleftrightarrow{B}, \overrightarrow{B}, \overleftarrow{B} \in \mathbb{R}^{2m \times 2m}$ be the matrices where $( \overleftrightarrow{B} )_{ef} = (B)_{ef}$ if $i_f \leftrightarrow_{G} t_f$ and $0$ otherwise,
$( \overrightarrow{B} )_{ef} = (B)_{ef}$ if $i_f \rightarrow_{G} t_f$ and $0$ otherwise,
and $( \overleftarrow{B} )_{ef} = (B)_{ef}$ if $i_f \leftarrow_{G} t_f$ and $0$ otherwise, respectively.
These notations lead $A_{\alpha} = \overleftrightarrow{A} + \alpha \overrightarrow{A} + \bar{\alpha} \overleftarrow{A}$
and $B_{\alpha} = \overleftrightarrow{B} + \alpha \overrightarrow{B} + \bar{\alpha} \overleftarrow{B}$ hold.

The number of mixed walks plays a crucial role in proving the results of this paper.
Let us define matrices $P_{(k,r)} \in \mathbb{Z}^{n\times n},\; k\geq 0, r\in \{ 0,\dots,R-1 \}$ whose $(u,v)$ entry $( P_{(k,r)} )_{u,v}$ is the number of NBT mixed walks in $G$ of length $k$ and rotation $r$, starting at $u$ and ending at $v$.
We set $P_{(0,0)} = I$ and $P_{(0,r)} = \bm{0},\; r\neq 0$, where $I$ represents the identity matrix.
The definitions of $\overleftrightarrow{A}, \overrightarrow{A}$ and $\overleftarrow{A}$ lead that $P_{(1,0)} = \overleftrightarrow{A},\; P_{(1,1)} = \overrightarrow{A},\; P_{(1,R-1)} = \overleftarrow{A},\; P_{(1,r)} = \bm{0},\; r=2,\dots, R-2$ are satisfied.
As similar quantities, let $n_{(k,r)} \in \mathbb{Z}$ be the number of NBT cycles of length $k$ and rotation $r$ without a tail.
We also define matrices $Q_{(k,r)} \in \mathbb{Z}^{2m\times 2m},\; k\geq 0, r\in \{ 0,\dots,R-1 \}$ whose $(e,f)$ entry $( Q_{(k,r)} )_{ef}$ represents the number of NBT mixed walks of length $k$ and rotation $r$ that end with $f$, to which the directed edge $e$ can be prepended while maintaining the NBT property.
To be more specific, let $W$ be any mixed walk counted in $( Q_{(k,r)} )_{ef}$.
The definition of $(Q_{(k,r)} )_{ef}$ implies that the directed edge $e$ can be connected to the beginning of $W$,
and the resulting mixed walk, formed by concatenating $e$ to the beginning of $W$, remains NBT.
We set $Q_{(0,0)} = I$ and $Q_{(0,r)} = \bm{0},\; r\neq 0$.
It is evident that $Q_{(1,0)} = \overleftrightarrow{B}$, $Q_{(1,1)} = \overrightarrow{B}$,
$Q_{(1,R-1)} = \overleftarrow{B}$ and $Q_{(1,r)} = \bm{0},\; r=2,\dots, R-2$ from the definition of $\overleftrightarrow{B}, \overrightarrow{B}$ and $\overleftarrow{B}$.

Since the NBT matrix represents the relation between edges, its eigenvector is indexed by edges.
To extract vertex-related information from eigenvectors, as will be discussed in Section~\ref{section:application},
we need to convert a vector indexed by edges to a vector indexed by vertices.
Such an approach has been adopted in \citep{Krzakala2013-up,Newman2013-zc}, and here we introduce an extension of the approach.
Assume that $g \in \mathbb{C}^{2m}$ is a $2m$ dimensional vector, where each index corresponds to a directed edge in $\vec{E}$.
We introduce the in- and out-vector of $g$, denoted by $g_{\alpha}^{\text{in}},\; g_{\alpha}^{\text{out}} \in \mathbb{C}^{n}$, respectively. The elements of $g_{\alpha}^{\text{in}}$ and $g_{\alpha}^{\text{out}}$ at vertex $u$ are defined as follows:
\begin{align}
	\left( g_{\alpha}^{\text{in}} \right)_{u}  & \coloneqq
	\sum_{v\in N_{u}} g_{\overrightarrow{vu}}
	,\label{definition:in-vector}                          \\
	\left( g_{\alpha}^{\text{out}} \right)_{u} & \coloneqq
	\sum_{v\in \overleftrightarrow{N_u}} g_{\overrightarrow{uv}}
	+ \alpha \sum_{v \in \overrightarrow{N_u}} g_{\overrightarrow{uv}}
	+ \bar{\alpha} \sum_{v \in \overleftarrow{N_u}} g_{\overrightarrow{uv}}
	.\label{definition:out-vector}
\end{align}

\subsection{Basic properties}

We show some features about the number of mixed walks.
These results extend the properties in~\citep{Stark1996-mx,Tarfulea2009-sy} to mixed walks in the sense that traversing a directed edge in the reverse direction is permitted.
Let $r_k \coloneqq \sum_{r=0}^{R-1} \alpha^r P_{(k,r)} \in \mathbb{C}^{n\times n}$.
The value $(r_k)_{u,v}$ is analogous to the number of NBT walks of length $k$ between two vertices $u$ and $v$ in an undirected graph.
However, unlike the undirected case, we must distinguish whether each directed edge in a mixed walk of length $k$ is traversed in the forward or reverse direction.
The entry $(r_k)_{u,v}$ represents the $\alpha^r$-weighted sum of the number of NBT mixed walks of length $k$ and rotation $r$ from $u$ to $v$, taken over all possible rotations.
With respect to $r_k$, the following lemma holds.
\begin{lemma}\label{lemma:r_k}
	Let $G$ be a directed graph, then we have
	\begin{align*}
		r_1 & = A_{\alpha}
		,\quad
		r_2 = A_{\alpha}^2 - D
		,                                          \\
		r_k & = r_{k-1} A_{\alpha} - r_{k-2} (D-I)
		,\quad k\geq 3
		.
	\end{align*}
\end{lemma}

\begin{proof}
	From the definition of $P_{(k,r)}$, it is clear that $r_1 = A_{\alpha}$.
	Next, to calculate $r_2$, we need to count NBT mixed walks.
	It is obvious that $P_{(2,3)} = \dots = P_{(2,R-3)} = \bm{0}$ since it is impossible to rotate more than twice in two steps.
	Thus, $r_2$ is calculated from $P_{(2,0)}, P_{(2,1)}, P_{(2,2)}, P_{(2,R-2)}$ and $P_{(2,R-1)}$.
	A mixed walk of length $2$ can be considered as an extension of a mixed walk of length $1$ by adding an edge.
	Therefore, a crucial point to consider is whether backtracking occurs when an edge is added.
	Mixed walks of length $2$ and rotation $0$ are composed of three types: adding an edge with rotation $0$ to a mixed walk of length $1$ and rotation $0$, adding an edge with rotation $-1$ to a mixed walk of length $1$ and rotation $1$, and adding an edge with rotation $1$ to a mixed walk of length $1$ and rotation $-1$.
	We note that these mixed walks contain backtracking mixed walks when the start and end points coincide.
	This is because a mixed walk of length $2$ consists of a sequence of edges that pass through three vertices.
	A backtracking mixed walk of length $2$ that starts and ends at vertex $u$ can be classified into one of three categories: (i) traversing a bidirectional edge to another vertex $v$ and returning to $u$ via the same edge, (ii) moving from $u$ to $v$ along an edge in the forward direction, then reversing and returning to $u$ via the same edge, or (iii) traversing a directed edge from $v$ to $u$ in the reverse direction, then proceeding in the forward direction back to $u$ via the same edge. Thus, $P_{(2,0)}$ can be expressed as $P_{(1,0)} \overleftrightarrow{A} + P_{(1,R-1)} \overrightarrow{A} + P_{(1,1)} \overleftarrow{A} - D$.
	Similarly, we find that
	\begin{align*}
		  & P_{(2,1)} =
		P_{(1,0)} \overrightarrow{A}
		+ P_{(1,1)} \overleftrightarrow{A}
		, & \quad
		  & P_{(2,2)} =
		P_{(1,1)} \overrightarrow{A}
		,                 \\
		  & P_{(2,R-1)} =
		P_{(1,0)} \overleftarrow{A}
		+ P_{(1,R-1)} \overleftrightarrow{A}
		, & \quad
		  & P_{(2,R-2)} =
		P_{(1,R-1)} \overleftarrow{A}
		.
	\end{align*}
	These results lead $r_2 = A_{\alpha}^2 - D$ from its definition.

	For $k \geq 3$, we similarly count the number of mixed walks of length $k-1$ extended by an edge, and then subtract the number of backtracking mixed walks. When considering $P_{(k,r)}$, there is a slight difference with $k=2$ in the way backtracking mixed walks are counted. The one-edge added mixed walks of length $k$ and rotation $r$ are expressed as $P_{(k-1,r)} \overleftrightarrow{A} + P_{(k-1,r-1)} \overrightarrow{A} + P_{(k-1,r+1)} \overleftarrow{A}$.
	This includes backtracking mixed walks, which are characterized by the absence of backtracking in the length $k-1$, but the addition of an edge induces backtracking.
	Namely, the backtracking mixed walks have the property that the vertex at length $k-2$ is the same as the vertex at length $k$.
	For each NBT mixed walk of length $k-2$ with rotation $r$, starting at vertex $u$ and ending at $v$, there exist $d_v-1$ such backtracking mixed walks. The subtraction of $1$ is necessary to prevent backtracking at length $k-1$. Thus,
	\begin{align}
		P_{(k,r)} = P_{(k-1,r)} \overleftrightarrow{A} + P_{(k-1,r-1)} \overrightarrow{A} + P_{(k-1,r+1)} \overleftarrow{A} - P_{(k-2,r)} (D - I)
		.\label{eq1:lemma:r_k}
	\end{align}
	This result holds in $r=0,\dots, R-1$.
	Substituting Eq.~\eqref{eq1:lemma:r_k} into the definition of $r_k$ proves the lemma.
\end{proof}

A similar relation that holds between $P_{(k,r)}$ and $\overleftrightarrow{A}, \overrightarrow{A}, \overleftarrow{A}$ also holds between $Q_{(k,r)}$ and $\overleftrightarrow{B}, \overrightarrow{B}, \overleftarrow{B}$.
\begin{lemma}\label{lemma:basic2}
	Let $G$ be a directed graph and $B_{\alpha}$ be the CNBT matrix defined in \eqref{definition:B_alpha}. Then, the following holds:
	\begin{align*}
		Q_{(k,r)} =
		\overleftrightarrow{B} Q_{(k-1,r)}
		+ \overrightarrow{B} Q_{(k-1,r-1)}
		+ \overleftarrow{B} Q_{(k-1,r+1)}
		,\quad
		k\geq 1
		.
	\end{align*}
\end{lemma}

\begin{proof}
	Since a NBT mixed walk in $( Q_{(k,r)} )_{ef}$ can be represented by adding an edge $h$ to a NBT mixed walk of length $k-1$, it can be described as follows:
	\begin{align*}
		\sum_{h} \left[
			( \overleftrightarrow{B} )_{eh} ( Q_{(k-1,r)} )_{hf}
			+ ( \overrightarrow{B} )_{eh} ( Q_{(k-1,r-1)} )_{hf}
			+ ( \overleftarrow{B} )_{eh} ( Q_{(k-1,r+1)} )_{hf}
			\right]
		.
	\end{align*}
	This does not contain backtracking mixed walks because added edge $h$ can be connected to a NBT mixed walk of length $k-1$ while maintaining NBT due to the definition of $Q$.
\end{proof}

\begin{corollary}\label{corollary:power_of_B_alpha}
	For $k\geq 1$, $\displaystyle ( B_{\alpha} )^k = \sum_{r=0}^{R-1} \alpha^r Q_{(k,r)}$.
\end{corollary}

\begin{proof}
	We prove Corollary~\ref{corollary:power_of_B_alpha} by induction.
	This corollary clearly holds for $k=1$.
	Suppose that this corollary holds when $k=k' \geq 1$. Then, we will show $( B_{\alpha} )^{k'+1} = \sum_{r=0}^{R-1} \alpha^{r} Q_{(k'+1,r)}$. From lemma~\ref{lemma:basic2},
	\begin{align*}
		\sum_{r=0}^{R-1} \alpha^{r} Q_{(k'+1,r)} & =
		\overleftrightarrow{B} \sum_{r=0}^{R-1} \alpha^{r} Q_{(k',r)}
		+ \alpha \overrightarrow{B} \sum_{r=0}^{R-1} \alpha^{r-1} Q_{(k',r-1)}
		+ \alpha^{-1} \overleftarrow{B} \sum_{r=0}^{R-1} \alpha^{r+1} Q_{(k', r+1)}
		,                                                                                                                                                                         \\
		                                         & \text{note that }
		\alpha^{-1} = \alpha^{R-1} = \bar{\alpha},\;
		\alpha^{R} = \alpha^{0} = 1,\;
		Q_{(k',-1)} = Q_{(k',R-1)},\;
		Q_{(k',R)} = Q_{(k',0)}
		\\
		                                         & = (\overleftrightarrow{B} + \alpha \overrightarrow{B} + \bar{\alpha} \overleftarrow{B}) \sum_{r=0}^{R-1} \alpha^{r} Q_{(k',r)}
		= ( B_{\alpha} )^{k'+1}
		.
	\end{align*}
	By induction, this result proves Corollary~\ref{corollary:power_of_B_alpha}.
\end{proof}

\section{Relations with the Hermitian adjacency matrix}\label{section:relation_with_herm_adj_mat}

In this section, we show that the newly introduced CNBT matrix is a natural extension of the conventional NBT matrix,
in the sense that it satisfies the relations between the traditional adjacency matrix and the NBT matrix.
This is supported by two results, Corollary~\ref{corollary:ihara_formula} and Theorem~\ref{theorem:edge-wise_to_node-wise}.

\subsection{Relation via the Ihara's formula}

We discuss the complex version of the Ihara's formula, which is a theorem that describes the zeta function of a graph can be expressed as a rational function, and it is also one of the equations that connects the adjacency matrix and the NBT matrix.
The Ihara zeta function on undirected graphs $Z_X(u)$ is defined as a product over all equivalence classes of NBT cycles without a tail~\citep{Stark1996-mx}:
\begin{align*}
	Z_{X}(u) \coloneqq \prod_{\substack{[C] \\\text{primitive}\\\text{NBT, no tail}}} (1 - u^{\left| C \right|})^{-1}
	.
\end{align*}
Corollary~\ref{corollary:ihara_formula} provides a variant of Ihara's formula, allowing reverse traversal of directed edges.
To derive this result, we use the weighted zeta function of directed graphs as follows~\citep{Konno2019-jd}:
\begin{align*}
	z_{X}(u; \alpha) \coloneqq \prod_{\substack{[C] \\\text{primitive}\\\text{NBT, no tail}}} (1 - \alpha^{r(C)} u^{\left| C \right|})^{-1}
	.
\end{align*}
In the zeta function for undirected graphs, equivalence classes of NBT cycles are formed based on cycle length.
In contrast, this paper introduces a concept of rotation for directed graphs, where the equivalence classes of cycles are determined by considering not only their lengths but also their rotations.
As a result, the rotation information is encoded as a weight in the weighted zeta function.

The Hermitian adjacency matrix defined as~\eqref{definition:A_alpha} is a type of the weighted matrix defined in~\citep{Konno2019-jd}.
Considering Theorem~3 in~\citep{Konno2019-jd}, the following result is derived.

\begin{corollary}\label{corollary:ihara_formula}
	Let $G$ be a directed graph on $n$ vertices and $m$ unoriented edges. Then,
	\begin{align*}
		\det (I-u B_{\alpha}) = (1-u^2)^{m-n} \det \left( I - u A_{\alpha} + u^2 (D - I) \right)
		.
	\end{align*}
\end{corollary}

\begin{proof}
	Our proof follows the general outline of the proof for Ihara's formula described in~\citep{Stark1996-mx}.
	We show that both sides are equal to the reciprocal of $z_{X}(u;\alpha)$.
	Since Theorem~3 in~\citep{Konno2019-jd} shows that the right-hand side is equal to the reciprocal to the zeta function,
	we prove $z_{X}(u; \alpha)^{-1} = \det (I - u B_{\alpha})$.

	%--- left-hand side
	Applying the Maclaurin expansion to $z_{X}(u;\alpha)$,
	\begin{align*}
		\log z_{X}(u; \alpha)
		= - \sum_{\substack{[C] \\\text{primitive}\\\text{NBT, no tail}}} \log \left( 1 - \alpha^{r(C)} u^{\left| C \right|} \right)
		= \sum_{\substack{[C]   \\\text{primitive}\\\text{NBT, no tail}}} \sum_{k\geq 1} \frac{\alpha^{k r(C)}}{k} u^{k \left| C \right|}
		.
	\end{align*}
	For $C_1, C_2 \in [C],\; r(C_1) = r(C_2)$ and $\#[C] = \left| C \right|,\; kr(C) = r(C^k), k\left| C \right| = |C^k|$ hold. Thus,
	\begin{align*}
		u \frac{d}{du} \log z_{X}(u; \alpha)
		= \sum_{\substack{[C] \\\text{primitive}\\\text{NBT, no tail}}} \sum_{k\geq 1} \left| C \right| \alpha^{k r(C)} u^{k \left| C \right|}
		= \sum_{\substack{C   \\\text{NBT, no tail}}} \alpha^{r(C)} u^{\left| C \right|}
		.
	\end{align*}
	Note that when partitioning all NBT cycles with no tail by length and rotation,
	the NBT cycles within the same group will have the same $r(C)$ and $|C|$. Thus,
	\begin{align}
		u \frac{d}{du} \log z_{X}(u; \alpha)
		 & = \sum_{k\geq 1} \sum_{r=0}^{R-1} \alpha^r n_{(k,r)} u^k
		.\label{eq1:corollary:ihara_formula}
	\end{align}

	%-- right-hand size
	With respect to $\det (I - u B_{\alpha})$,
	the definition of matrix logarithm and $\exp\left[ \mathrm{tr}\left( M \right) \right] = \det e^M$ for a matrix $M$ derive:
	\begin{align*}
		\log\det (I - u B_{\alpha})
		= - \sum_{k \geq 1} \frac{u^k}{k} \mathrm{tr}\left( B_{\alpha}^k \right)
		.
	\end{align*}
	Applying corollary~\ref{corollary:power_of_B_alpha} implies
	\begin{align*}
		u \frac{d}{du} \log\det \left( I - u B_{\alpha} \right)
		= - \sum_{k\geq 1} \sum_{r=0}^{R-1} \alpha^r \left( \sum_{e \in \vec{E}} \left( Q_{(k,r)} \right)_{ee} \right) u^k
		.
	\end{align*}
	The sum $\sum_{e \in \vec{E}} \left( Q_{(k,r)} \right)_{ee}$ counts the number of NBT cycles.
	For each cycle, the initial and terminal vertex is $t_e$, the last edge is $e$, and the initial edge can be connected to the trailing edge $e$ while maintaining NBT.
	This means the cycles don't have a tail and $\sum_{e \in \vec{E}} \left( Q_{(k,r)} \right)_{ee} = n_{(k,r)}$.
	From Eq.~\eqref{eq1:corollary:ihara_formula},
	\begin{align*}
		u \frac{d}{du} \log\det \left( I - u B_{\alpha} \right)
		 & = u \frac{d}{du} \log z_{X}(u; \alpha)^{-1}
		,                                              \\
		\therefore
		z_{X}(u; \alpha)^{-1}
		 & = \det \left( I - u B_{\alpha} \right)
		.
	\end{align*}
\end{proof}

\subsection{Relation via in/out-vectors}
As will be discussed in Section~\ref{section:application},
it is significant that in/out-vectors can be efficiently obtained from the eigenvectors of the CNBT matrix in applications.
Theorem~\ref{theorem:edge-wise_to_node-wise} is an analogous result that is known to hold between the commonly known adjacency matrix and the NBT matrix of undirected graphs~\citep{Krzakala2013-up}.

\begin{theorem}\label{theorem:edge-wise_to_node-wise}
	For any $g \in \mathbb{C}^{2m}$, the following holds:
	\begin{align}
		\begin{pmatrix}
			\left( B_{\alpha} g \right)_{\alpha}^{\text{out}}
			\\
			\left( B_{\alpha} g \right)_{\alpha}^{\text{in}}
		\end{pmatrix} =
		\begin{pmatrix}
			A_{\alpha} & -I
			\\
			D-I        & \bm{0}
		\end{pmatrix} \begin{pmatrix} g_{\alpha}^{\text{out}} \\ g_{\alpha}^{\text{in}} \end{pmatrix}
		.\label{eq:theorem:edge-wise_to_node-wise}
	\end{align}
\end{theorem}

\begin{proof}
	Note that $( \left( B_{\alpha} g \right)_{\alpha}^{\text{out}} )_{u}$ can be written as follows based on the definition of \eqref{definition:out-vector}.
	\begin{align*}
		\left( \left( B_{\alpha} g \right)_{\alpha}^{\text{out}} \right)_{u} & =
		\sum_{v \in \overleftrightarrow{N_u}} \left( B_{\alpha} g \right)_{\overrightarrow{uv}}
		+ \alpha \sum_{v \in \overrightarrow{N_u}} \left( B_{\alpha} g \right)_{\overrightarrow{uv}}
		+ \bar{\alpha} \sum_{v \in \overleftarrow{N_u}} \left( B_{\alpha} g \right)_{\overrightarrow{uv}}
		.
	\end{align*}
	We can calculate each terms as follows:
	\begin{align*} \left\{\begin{aligned}
			       \sum_{v \in \overleftrightarrow{N_u}} \left( B_{\alpha} g \right)_{\overrightarrow{uv}}
			        & =  \sum_{v \in \overleftrightarrow{N_u}} \left( g_{\alpha}^{\text{out}} \right)_{v}
			       - \sum_{v \in \overleftrightarrow{N_u}} g_{\overrightarrow{vu}}
			       ,                                                                                              \\
			       \alpha \sum_{v \in \overrightarrow{N_u}} \left( B_{\alpha} g \right)_{\overrightarrow{uv}}
			        & =  \alpha \sum_{v \in \overrightarrow{N_u}} \left( g_{\alpha}^{\text{out}} \right)_{v}
			       - \sum_{v \in \overrightarrow{N_u}} g_{\overrightarrow{vu}}
			       ,                                                                                              \\
			       \bar{\alpha} \sum_{v \in \overleftarrow{N_u}} \left( B_{\alpha} g \right)_{\overrightarrow{uv}}
			        & =  \bar{\alpha} \sum_{v \in \overleftarrow{N_u}} \left( g_{\alpha}^{\text{out}} \right)_{v}
			       - \sum_{v \in \overleftarrow{N_u}} g_{\overrightarrow{vu}}
			       .
		       \end{aligned}\right.
	\end{align*}
	This leads to $( \left( B_{\alpha} g \right)_{\alpha}^{\text{out}} )_{u} = ( A_{\alpha} g_{\alpha}^{\text{out}} )_{u} - ( g_{\alpha}^{\text{in}} )_{u}$.

	Similarly, $( \left( B_{\alpha} g \right)_{\alpha}^{\text{in}} )_{v}$ can be expressed according to the definition of \eqref{definition:in-vector} as follows,
	and each term can be calculated as shown below.
	\begin{align*}
		\left( \left( B_{\alpha} g \right)_{\alpha}^{\text{in}} \right)_{v} & =
		\sum_{u \in \overleftrightarrow{N_v}} \left( B_{\alpha} g \right)_{\overrightarrow{uv}}
		+ \sum_{u \in \overleftarrow{N_v}} \left( B_{\alpha} g \right)_{\overrightarrow{uv}}
		+ \sum_{u \in \overrightarrow{N_v}} \left( B_{\alpha} g \right)_{\overrightarrow{uv}}
		,                                                                                                                                                                                        \\
		                                                                    & \left\{\begin{aligned}
			                                                                             \sum_{u \in \overleftrightarrow{N_v}} \left( B_{\alpha} g \right)_{\overrightarrow{uv}} & =
			                                                                             \sum_{u \in \overleftrightarrow{N_v}} \left( g_{\alpha}^{\text{out}} \right)_{v}
			                                                                             - \sum_{u \in \overleftrightarrow{N_v}} g_{\overrightarrow{vu}}
			                                                                             ,                                                                                           \\
			                                                                             \sum_{u \in \overleftarrow{N_v}} \left( B_{\alpha} g \right)_{\overrightarrow{uv}}      & =
			                                                                             \sum_{u \in \overleftarrow{N_v}} \left( g_{\alpha}^{\text{out}} \right)_{v}
			                                                                             - \bar{\alpha} \sum_{u \in \overleftarrow{N_v}} g_{\overrightarrow{vu}}
			                                                                             ,                                                                                           \\
			                                                                             \sum_{u \in \overrightarrow{N_v}} \left( B_{\alpha} g \right)_{\overrightarrow{uv}}     & =
			                                                                             \sum_{u \in \overrightarrow{N_v}} \left( g_{\alpha}^{\text{out}} \right)_{v}
			                                                                             - \alpha \sum_{u \in \overrightarrow{N_v}} g_{\overrightarrow{vu}}
			                                                                             .
		                                                                             \end{aligned}\right.
	\end{align*}
	This leads to $( \left( B_{\alpha} g \right)_{\alpha}^{\text{in}} )_{v} = (d_v - 1) ( g_{\alpha}^{\text{out}} )_{v}$. From these results, we obtain the theorem.
\end{proof}

\begin{corollary}
	Let $(\lambda, \bm{u}) \in \mathbb{C}\times \mathbb{C}^{2m}$ be an eigenpair for $B_{\alpha}$. Then,
	$(\lambda, \begin{pmatrix} \bm{u}_{\alpha}^{\text{out}} \\ \bm{u}_{\alpha}^{\text{in}} \end{pmatrix})$ is an eigenpair for the matrix
	described in the right-hand side of~\eqref{eq:theorem:edge-wise_to_node-wise}.
\end{corollary}
This is evident from Theorem~\ref{theorem:edge-wise_to_node-wise}.
We can efficiently obtain the transformed eigenvectors of $B_{\alpha}$ from the adjacency-related matrix
described in the right-hand side of \eqref{eq:theorem:edge-wise_to_node-wise} since the matrix is generally smaller than $B_{\alpha}$.

\section{Application in clustering of directed graphs}\label{section:application}

In this section, we demonstrate the utility of the CNBT matrix for node clustering in directed graphs,
especially when the underlying graph is sparse.
A graph is said to be sparse if the number of edges is at most a constant multiple of the number of vertices.
In the context of undirected graphs,
spectral methods based on the adjacency matrix or on the derived graph Laplacians are widely used across various fields.
However, it is known that in sparse graphs, methods such as belief propagation~\citep{Yedidia2003-nz} can perform well,
whereas the spectral methods fail to estimate clusters accurately~\citep{Decelle2011-ir}.
To resolve this problem, several approaches have been proposed, including methods based on the NBT matrix~\citep{Krzakala2013-up,Saade2014-vg,Newman2013-zc}.
Regarding the detectability condition based on the edge probabilities between and within clusters,
methods using the NBT matrix have been experimentally shown to perform similarly to belief propagation and outperform methods based on the adjacency matrix~\citep{Krzakala2013-up}.

Recently, the problem of directed graph clustering has received significant attention,
highlighting the differences in problem compared to undirected graphs.
Unlike undirected graph clustering, which focuses on dense intra-cluster and sparse inter-cluster connections,
directed graph clustering considers the imbalance of directional information between clusters.
Hermitian adjacency matrices are often used to address this problem due to its symmetry and ease of handling.
The approach has been shown to perform well in~\citep{Cucuringu2020-de,Laenen2020-ft,Martin2024-en}.

We illustrate that the similar issue of spectral methods based on the adjacency matrix failing to work for sparse undirected graphs also arises in directed graphs.
We experimentally show that a method based on a Hermitian adjacency matrix fail to resolve this issue,
whereas the proposed method can mitigate the problem.

\subsection{Our algorithm}

\begin{algorithm}[tb]
	\caption{CNBT-SC~(in/out)}
	\label{algorithm:cnbt_sc}
	\begin{algorithmic}[1]
		\Require{a digraph $G=(V,E)$, the number of clusters $K$}
		\Ensure{clusters $S_1, \dots S_K \subset V$ with $S_i \cap S_j = \phi,\; i\neq j$}
		\State $\alpha \gets \exp\left( \frac{2\pi \sqrt{-1}}{K} \right)$
		\State Compute the CNBT matrix $B_{\alpha}$
		\State Compute the eigenvectors $g_1,\dots, g_{\lfloor \frac{K}{2} \rfloor} \in \mathbb{C}^{2m}$ corresponding to the eigenvalues with the largest real part of $B_{\alpha}$
		\State Compute $(g_1)_{\alpha}^{\text{out}}, \dots, (g_{\lfloor \frac{K}{2} \rfloor})_{\alpha}^{\text{out}} \in \mathbb{C}^{n}$ (or compute in-vectors)
		\State $X \gets \begin{pmatrix} (g_1)_{\alpha}^{\text{out}} & \dots & (g_{\lfloor \frac{K}{2} \rfloor})_{\alpha}^{\text{out}} \end{pmatrix} \in \mathbb{C}^{n \times \lfloor \frac{K}{2} \rfloor}$ (or apply to in-vectors)
		\State Normalize each row of $X$ and denote the result as $\tilde{X}$
		\State Use K-means algorithm to $\begin{pmatrix} \text{real}(\tilde{X}) & \text{imag}(\tilde{X}) \end{pmatrix}$
	\end{algorithmic}
\end{algorithm}

We propose a new spectral algorithm~(\textit{CNBT-SC}, Algorithm~\ref{algorithm:cnbt_sc}) using the CNBT matrix.
We set $\alpha=\exp\left( \frac{2\pi \sqrt{-1}}{K} \right)$ based on experimental trials and the prior research~\citep{Laenen2020-ft}.
Furthermore, we utilize $\lfloor \frac{K}{2} \rfloor$ eigenvectors, treating the real and imaginary components separately to obtain a final dimensionality of $K$.
The codes implementing our algorithm and experiments will be made available in a public repository upon acceptance.

\subsection{Experimental setup}

We present two experiments on directed graphs generated from the stochastic block models.
The first experiment compares the proposed algorithm with existing methods under the same experimental conditions~\citep{Cucuringu2020-de}.
The directed stochastic block model~(DSBM, \citep{Cucuringu2020-de}) is a generative model for directed graphs with cluster structure, parameterized by $K, n, p, \eta$ and $F \in \mathbb{R}^{K\times K}$. A generated graph consists of $n=5000$ vertices divided into $K=5$ clusters of the same size. Each pair of vertices $(u,v)$ is connected by a directed edge with probability $p$. The direction of the edge is determined as follows: given that vertices $u$ and $v$ belong to cluster $a$ and $b$, respectively, the edge is oriented from $u$ to $v$ with probability $F_{a,b}$. The matrix $F$ satisfies the constraint $F_{a,b} + F_{b,a} = 1$. In this experiment, we assume a circular pattern between clusters and define the elements of $F$ as follows:
\begin{align*}
	F_{a,b} \coloneqq \begin{cases}
		                  1 - \eta & a+1 \equiv b \text{ mod } K
		                  ,                                      \\
		                  \eta     & b+1 \equiv a \text{ mod } K
		                  ,                                      \\
		                  0.5      & \text{otherwise}
		                  .
	                  \end{cases}
\end{align*}
The parameter $\eta$ controls the strength of the circular pattern. Smaller values of $\eta$ enhance the circular pattern, whereas values approaching $0.5$ gradually diminish the cluster structure.

The second experiment investigates the influence of two parameters, $\epsilon$ and $\eta$ under sparse graphs.
The parameter $\epsilon$ controls the strength of the circular pattern, and $\eta$ governs the ratio of intra-cluster edge probability to inter-cluster edge probability. In addition to the DSBM used in the first experiment, we also use the degree-corrected stochastic block model~(DCSBM, \citep{Karrer2011-ic}), which tends to generate a few number of hub vertices.
A generated graph consists of $n=3000$ vertices divided into $K=3$ clusters of the same size.
Given that vertices $u$ and $v$ belong to cluster $a$ and $b$, respectively, a directed edge from $u$ to $v$ is generated with probability $\frac{\Gamma_{a,b}}{n}$ on DSBM.
We note that the inclusion of $n$ in the denominator of the edge probability is the primary factor in generating sparse graphs.
To control a circular pattern, we set $\Gamma_{a,b}$ as follows:
\begin{align*}
	\Gamma_{a,b} \coloneqq \begin{cases}
		                       \Gamma_{\text{correct}} & a+1 \equiv b \text{ mod } K
		                       ,                                                     \\
		                       \Gamma_{\text{reverse}} & b+1 \equiv a \text{ mod } K
		                       ,                                                     \\
		                       \Gamma_{\text{intra}}   & \text{otherwise}
		                       .
	                       \end{cases}
	\quad
	\epsilon \coloneqq \frac{\Gamma_{\text{correct}}}{\Gamma_{\text{reverse}}}
	,\quad
	\eta \coloneqq \frac{\Gamma_{\text{reverse}}}{\Gamma_{\text{intra}}}
	.
\end{align*}
We set $\Gamma_{\text{correct}},\; \Gamma_{\text{reverse}}$ and $\Gamma_{\text{intra}}$ such that the expected degree $c$ of the generated graph is $5$.
Given $\epsilon$ and $\eta$, these three parameters are computed as
\begin{align*}
	% (2c * ϵ * η) / (1 + η * (1 + ϵ))
	\Gamma_{\text{correct}} = \frac{2c \epsilon \eta}{1 + \eta (1 + \epsilon)}
	,\quad
	% (2c * η) / (1 + η * (1 + ϵ))
	\Gamma_{\text{reverse}} = \frac{2c \eta}{1 + \eta (1 + \epsilon)}
	,\quad
	%2c / (1 + η * (1 + ϵ))
	\Gamma_{\text{intra}} = \frac{2c}{1 + \eta (1 + \epsilon)}
	.
\end{align*}
Under the DCSBM model, we assume that a directed edge from $u$ to $v$ is generated with probability $\frac{\theta_u \theta_v \Gamma_{a,b}}{n}$.
The parameter $\theta_u \in \mathbb{R}$ allows for heterogeneous expected degrees across vertices.
We assume $\theta_u$ follows a Pareto distribution with scale $1$ and exponent $-2.5$,
and $\Gamma$ is derived using the same formula as in the DSBM.

We compare the performance of our algorithm with the other spectral clustering algorithms, a variant of DI-SIM~\citep{Rohe2016-io}, BI-SYM and its normalization version~(DD-SYM)~\citep{Satuluri2011-ja}, Herm~\citep{Cucuringu2020-de}, and SimpleHerm~\citep{Laenen2020-ft}.
We evaluate the estimation accuracy using the Adjusted Rand Index~(ARI, \citep{Rand1971-iq,Hubert1985-bg}) which measures the similarity between two clusters.
Thus, the higher the value, the more similar the two clusters are.

\subsection{Experimental results}

\begin{figure}[tb]
	\begin{minipage}{.23\textwidth}
		\centering
		\includegraphics[width=1\textwidth]{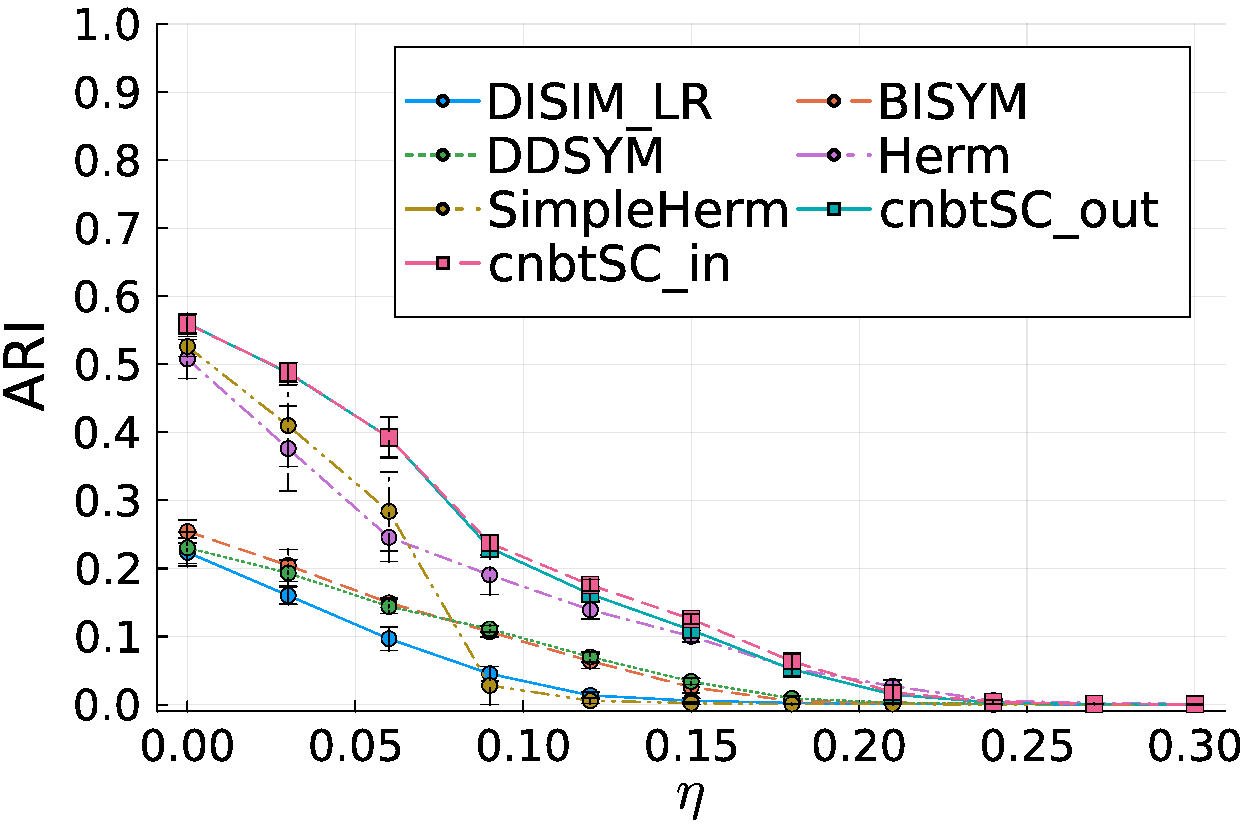}
		\subcaption{$p=0.45\%$}
	\end{minipage}
	\hfill
	\begin{minipage}{.23\textwidth}
		\centering
		\includegraphics[width=1\textwidth]{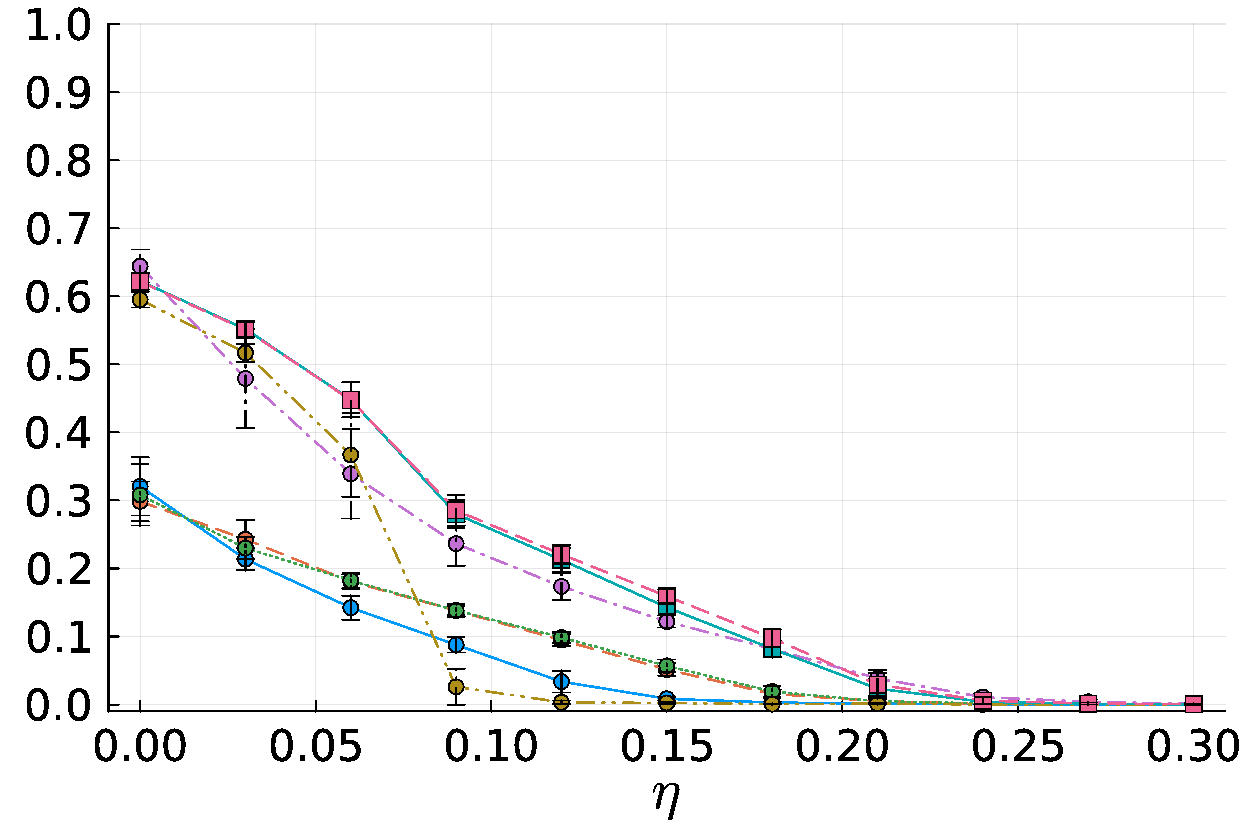}
		\subcaption{$p=0.5\%$}
	\end{minipage}
	\hfill
	\begin{minipage}{.23\textwidth}
		\centering
		\includegraphics[width=1\textwidth]{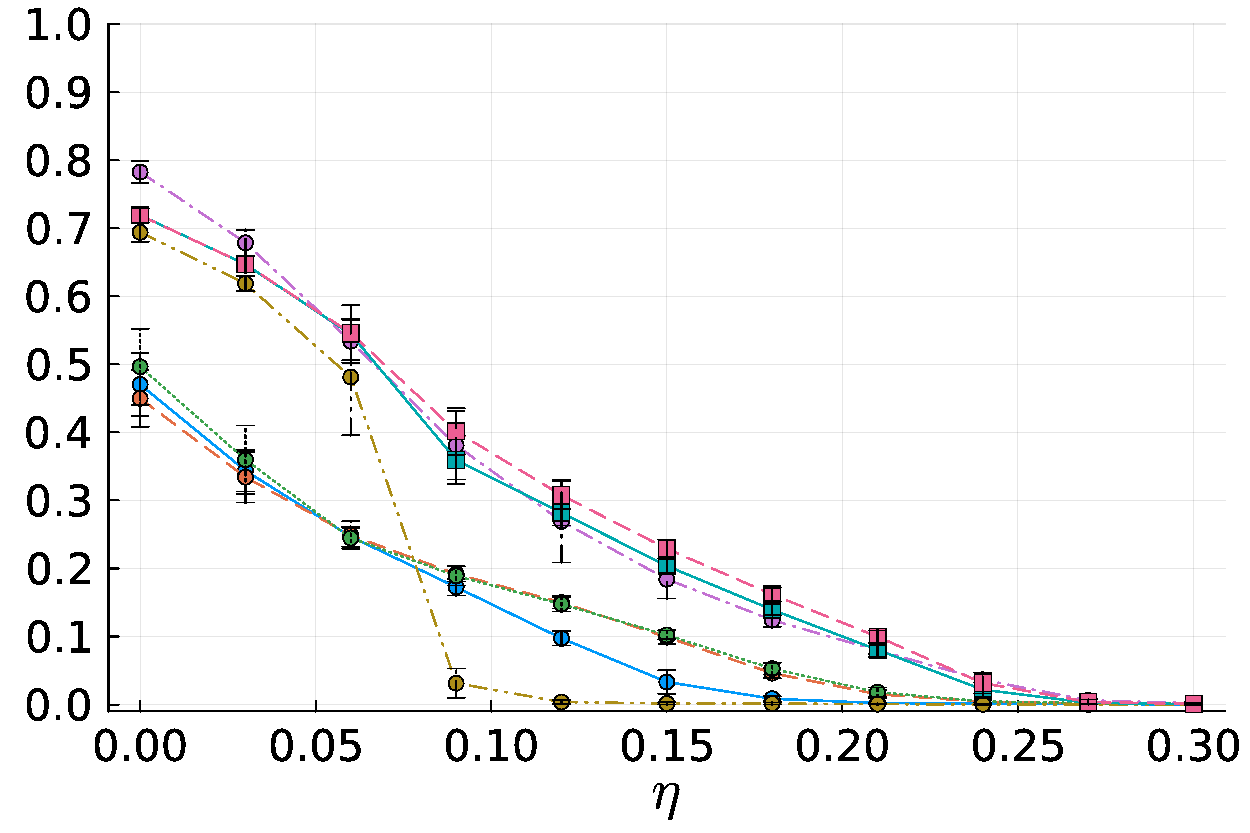}
		\subcaption{$p=0.6\%$}
	\end{minipage}
	\hfill
	\begin{minipage}{.23\textwidth}
		\centering
		\includegraphics[width=1\textwidth]{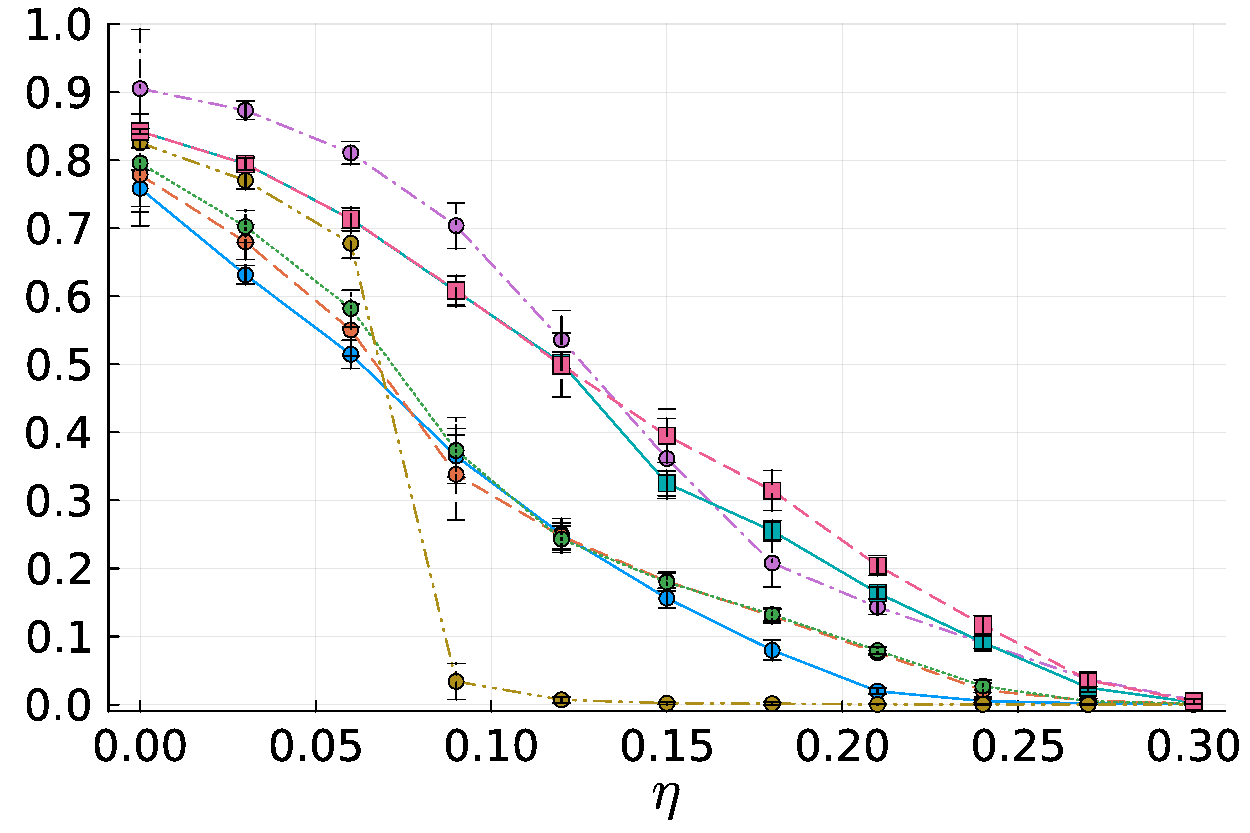}
		\subcaption{$p=0.8\%$}
	\end{minipage}
	\caption{Recovery rates of the first experiment. The results show the average ARI over $10$ independent simulations.}
	\label{figure:1st_experiment}
\end{figure}

\begin{figure}[tb]
	\begin{minipage}{.45\textwidth}
		\centering
		\includegraphics[width=1\textwidth]{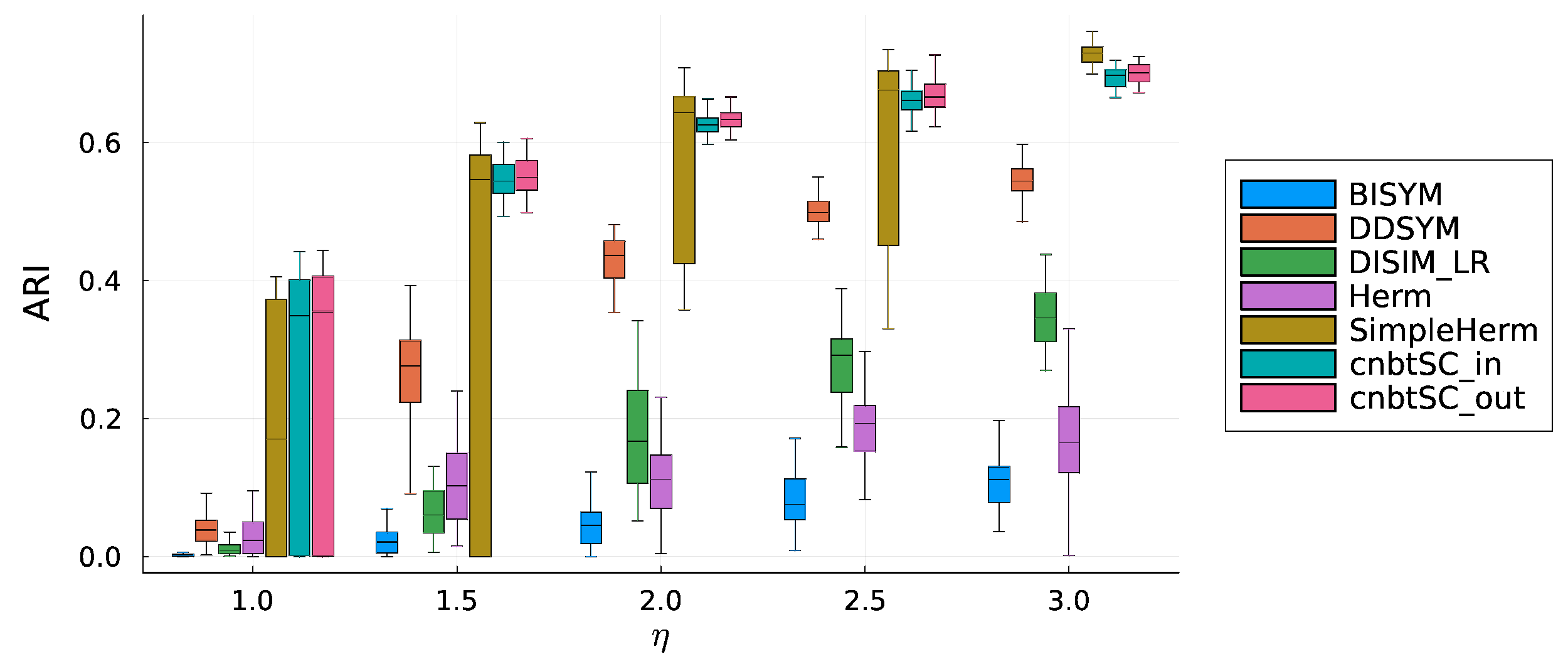}
		\subcaption{DSBM with $\epsilon = 4$}
		\label{figure:dsbm_eps}
	\end{minipage}
	\hfill
	\begin{minipage}{.45\textwidth}
		\centering
		\includegraphics[width=1\textwidth]{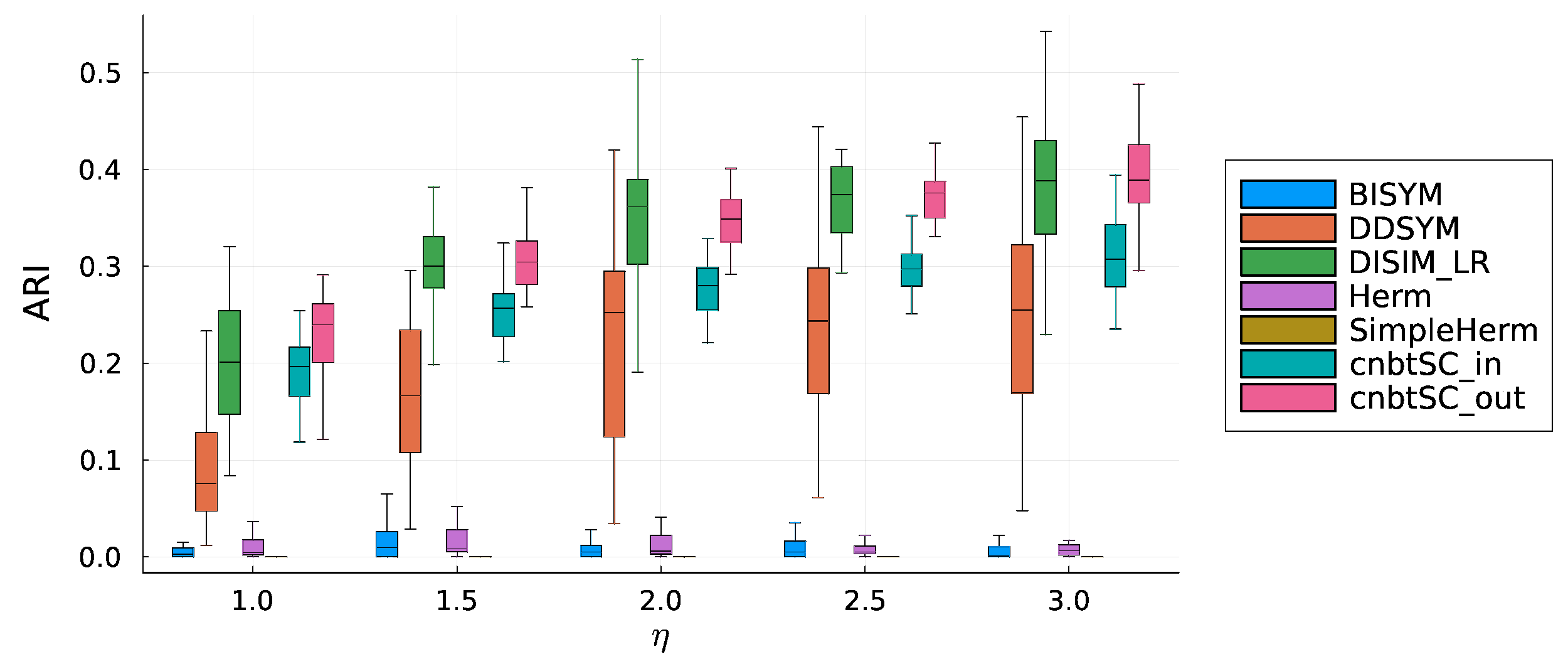}
		\subcaption{DCSBM with $\epsilon = 4$}
		\label{figure:dcsbm_eps}
	\end{minipage}
	\begin{minipage}{1\textwidth}
		\begin{minipage}{.3\textwidth}
			\centering
			\includegraphics[width=1\textwidth]{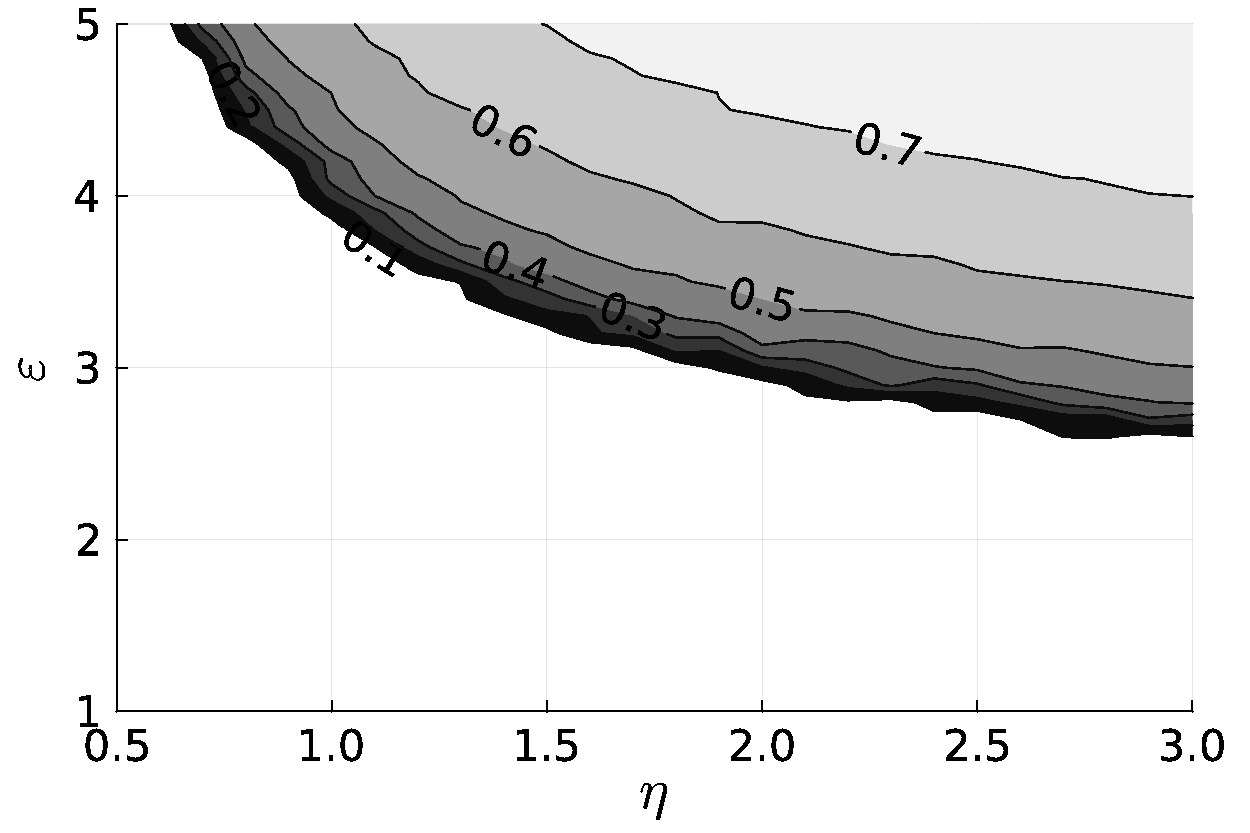}
		\end{minipage}
		\hfill
		\begin{minipage}{.3\textwidth}
			\centering
			\includegraphics[width=1\textwidth]{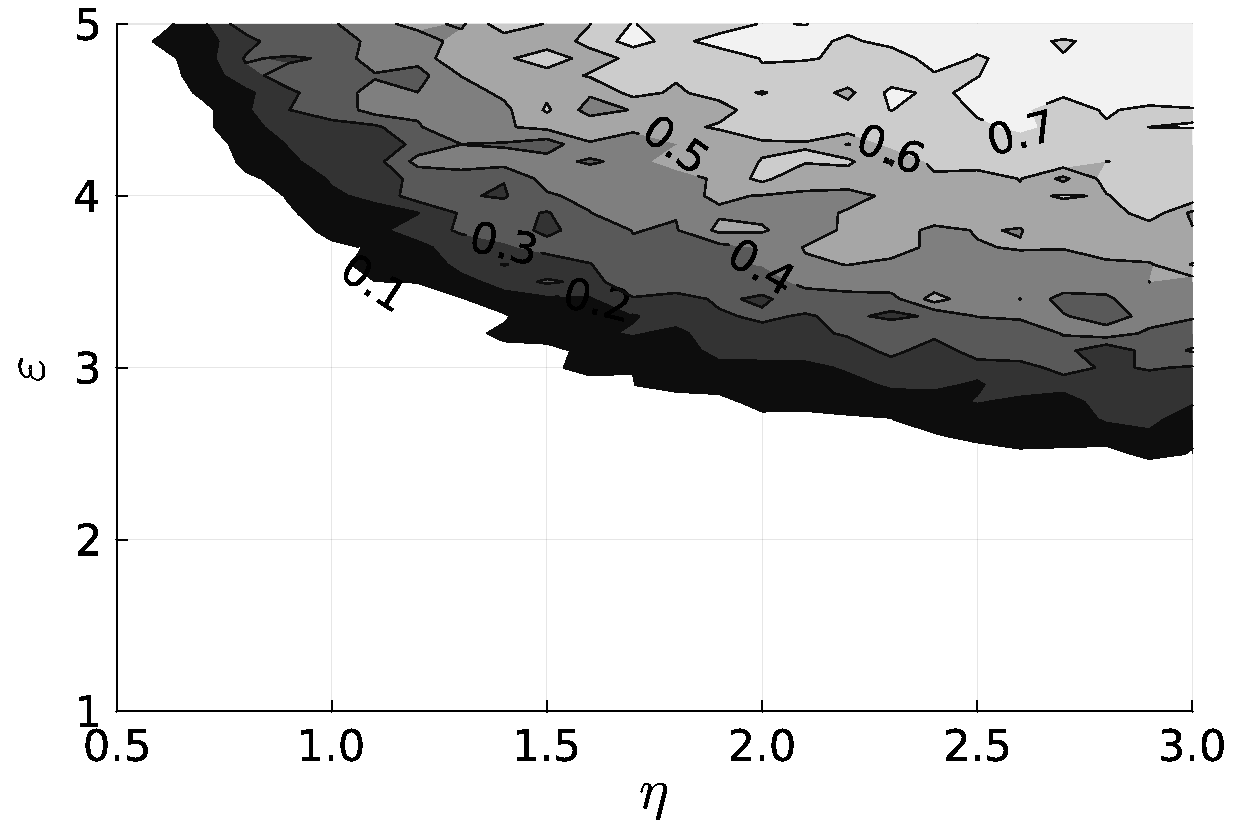}
		\end{minipage}
		\hfill
		\begin{minipage}{.3\textwidth}
			\centering
			\includegraphics[width=1\textwidth]{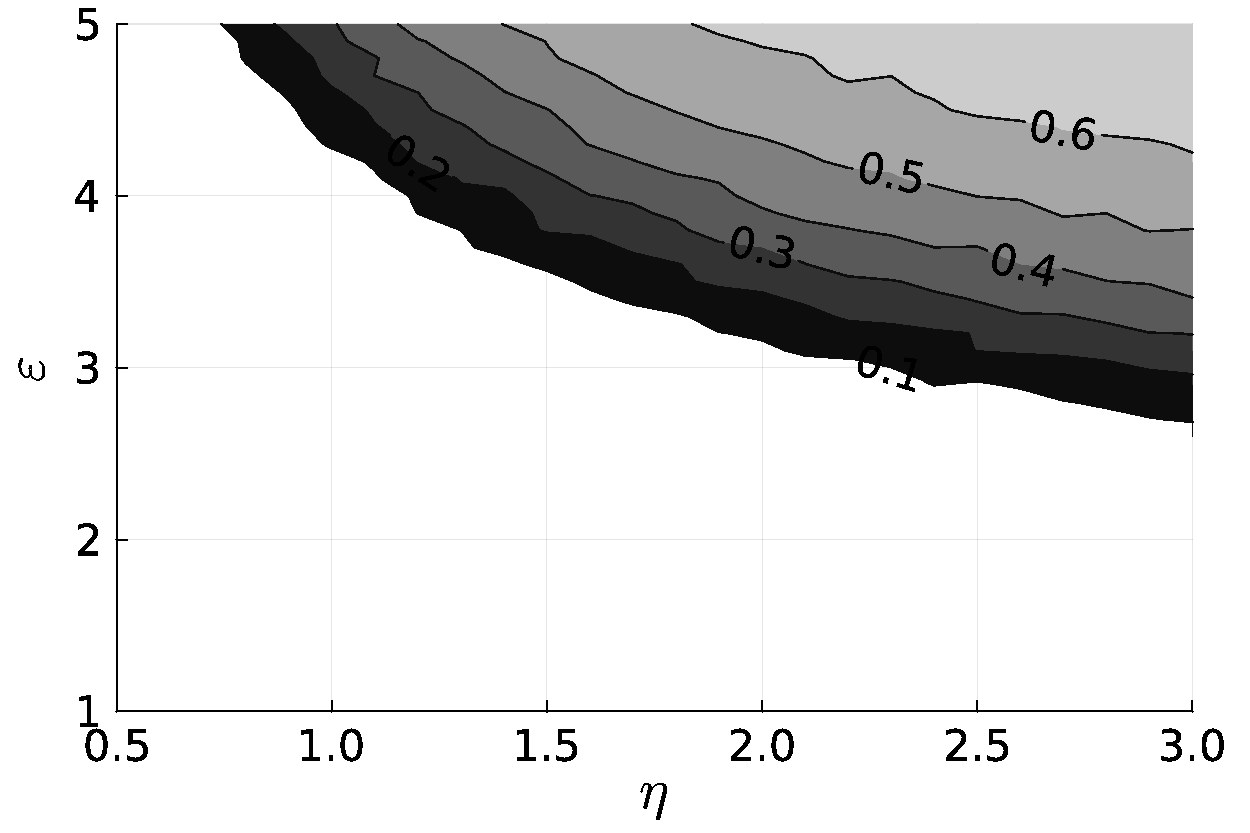}
		\end{minipage}
		\subcaption{Contour plots over $\epsilon$ and $\eta$ under the DSBM, showing the top three methods from Figure~\ref{figure:dsbm_eps} (from left to right: CNBT-SC~(out), SimpleHerm, DD-SYM)}
	\end{minipage}
	\begin{minipage}{1\textwidth}
		\begin{minipage}{.3\textwidth}
			\centering
			\includegraphics[width=1\textwidth]{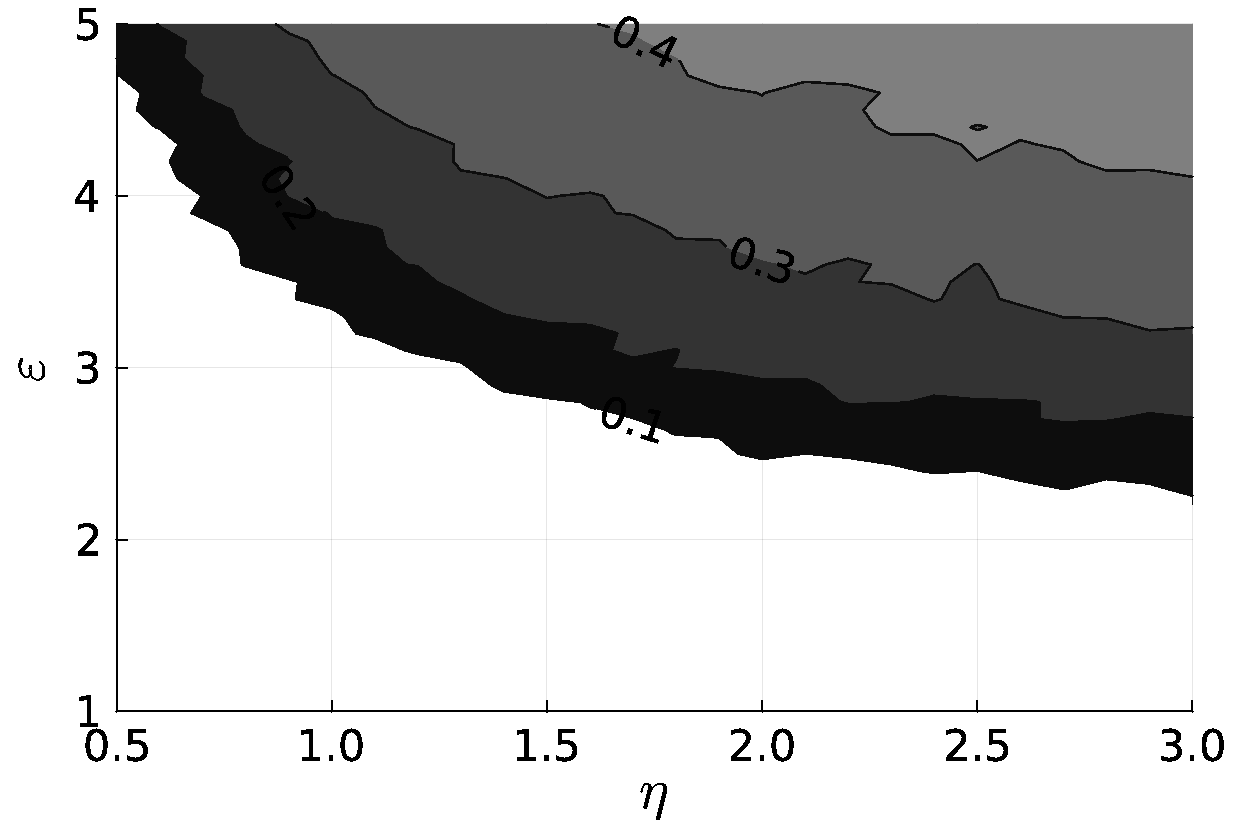}
		\end{minipage}
		\hfill
		\begin{minipage}{.3\textwidth}
			\centering
			\includegraphics[width=1\textwidth]{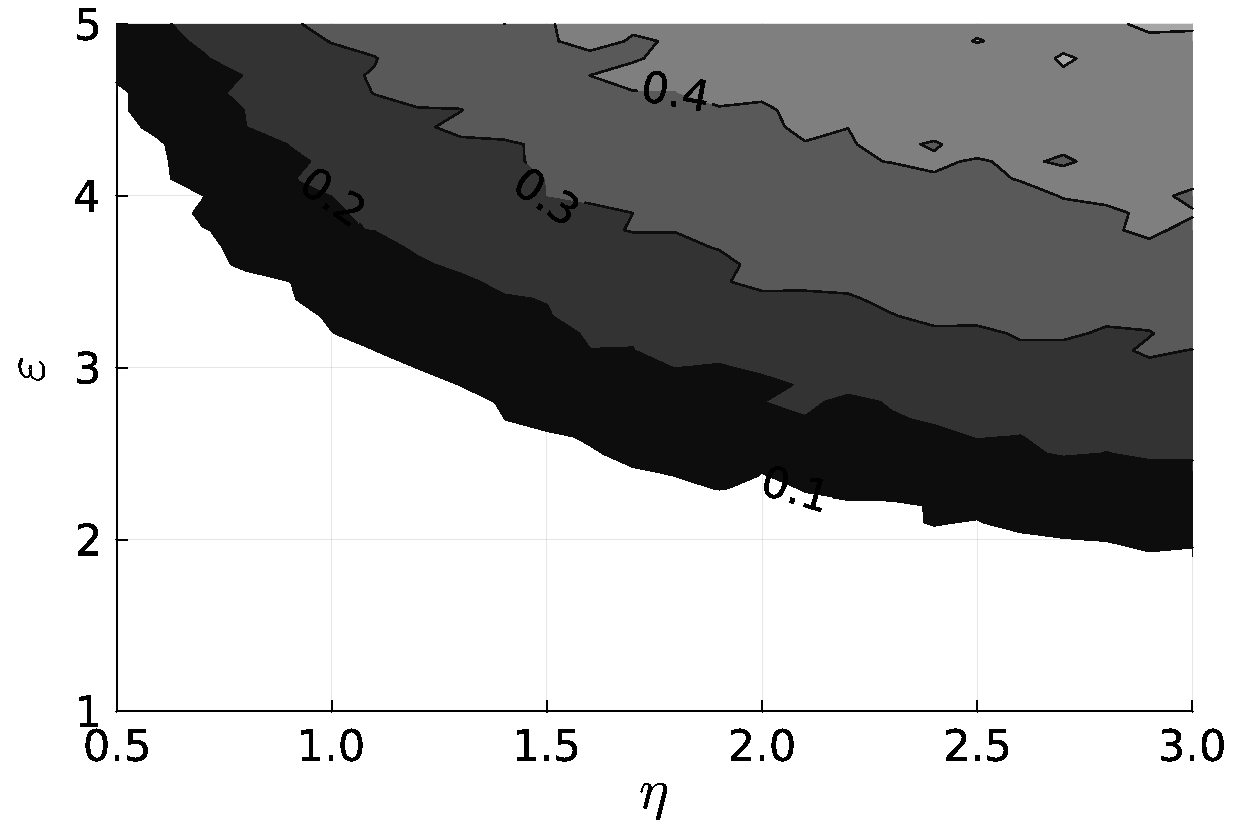}
		\end{minipage}
		\hfill
		\begin{minipage}{.3\textwidth}
			\centering
			\includegraphics[width=1\textwidth]{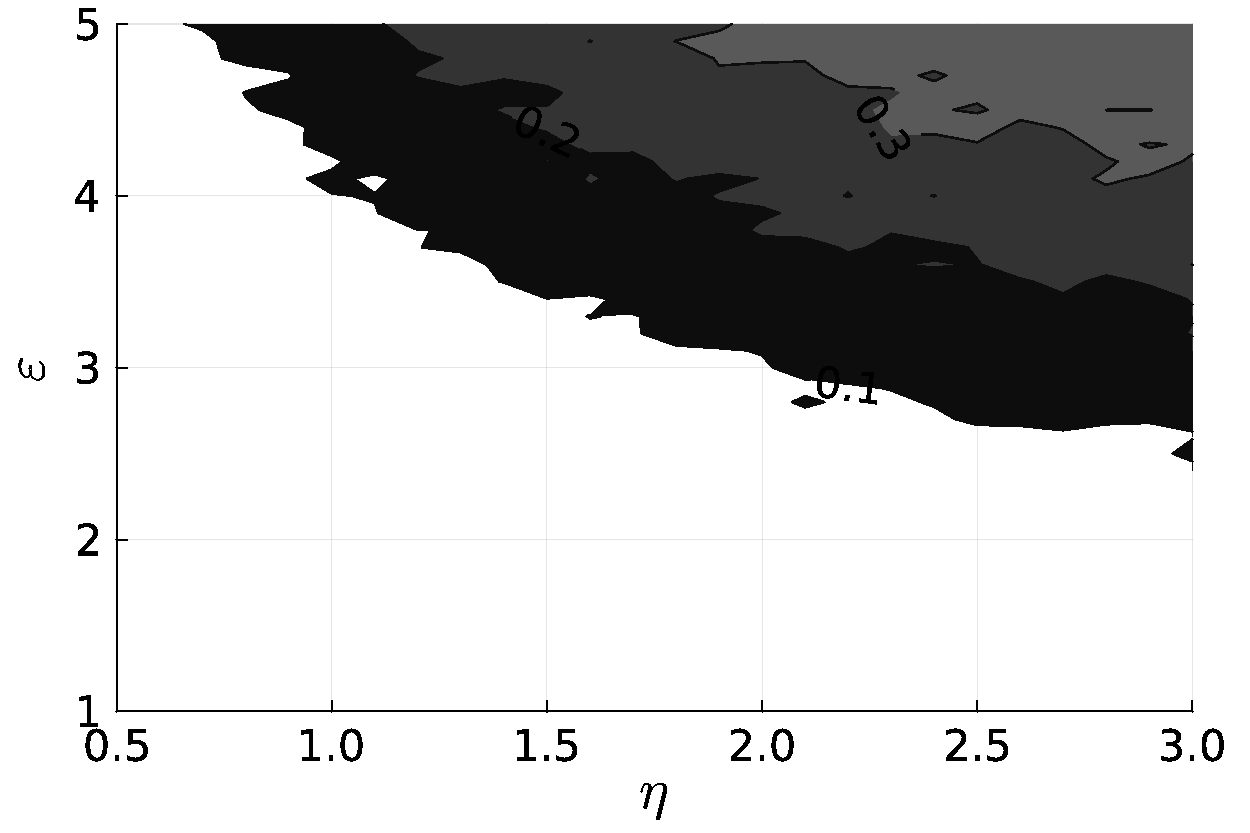}
		\end{minipage}
		\subcaption{Contour plots over $\epsilon$ and $\eta$ under the DCSBM, showing the top three methods from Figure~\ref{figure:dcsbm_eps} (from left to right: CNBT-SC~(out), DI-SIM, DD-SYM)}
	\end{minipage}
	\caption{Comparison of the second experiment. The top row shows boxplots with $\epsilon=4$. Contour lines represent the average ARI over $30$ independent trials.}
	\label{figure:2nd_experiment}
\end{figure}

Figure~\ref{figure:1st_experiment} shows the result of the first experiment. While SimpleHerm exhibits good performance when $\eta$ is small, it is observed to be the most sensitive to the increase of $\eta$, with its performance degrading more rapidly than other methods.
The proposed algorithm achieves competitive performance in all scenarios, and it shows the most superior results especially when $p$ is small.

The results of the second experiment are shown in Figure~\ref{figure:2nd_experiment}.
For most methods, we observe a general trend of improved performance as the correct circular pattern becomes more pronounced~(i.e., with increasing $\epsilon$), and as the ratio of inter-cluster edges grows~(i.e., which increasing $\eta$).
It is noteworthy that the proposed method consistently achieves good results on both DSBM and DCSBM,
while the performance of other methods deteriorates on either one of these models.
The stable performance of DD-SYM across both models likely stems from its normalization process,
which was specifically designed to handle the presence of hub vertices.
Additionally, this result shows that SimpleHerm exhibits high variance in its performance under DSBM with sparse graphs.
Further investigation of individual trials reveals that occasionally, the eigenvector used by SimpleHerm failed to identify any cluster structure, which consequently led to an ARI value of approximately zero.

%%%%%%%%%%%%%%%%%%%%%%%%%%%%%%%%%%%%%%%%%%%%%%%%%%%
\subsection{Relationship with Belief Propagation}

It has already been demonstrated in \citep{Krzakala2013-up,Decelle2011-ir} that the NBT matrix naturally emerges from the linearization of the belief propagation~(BP)~\citep{Andersen1991-go,Mezard2009-ku} update equations. We investigate the idea to elucidate the relationship between the CNBT matrix and clustering.
A detailed derivation is described in appendix~\ref{appendix:derivation_delta}.

Let us consider the stochastic block models~(SBM) with $K$ clusters. Let $\psi_u$ denotes the prior distribution indicating which cluster a vertex $u$ belongs to, $t_u$ be a random variable representing the cluster of the vertex $u$, and $P_{ab}$ be the probability of generating a directed edge from a vertex in cluster $a$ to a vertex in cluster $b$. Then, the joint probability of $\left\{ t_u \right\}_{u \in V}$ and $A$ can be described as follows:
\begin{align*}
	f(\left\{ t_u \right\}_{u \in V}, A) \propto \left( \prod_{u \in V} \psi_u(t_u) \right) \left( \prod_{u\neq v} P_{t_u t_v}^{A_{uv}} \left(1 - P_{t_u t_v}\right)^{1-A_{uv}} \right)
	.
\end{align*}
Let a network be sparse by defining the edge probability as $P_{ab} = \frac{c_{ab}}{N}$,
where $c_{ab}$ is a constant and $N$ is the number of vertices.
Then, as described in \citep{Decelle2011-ir},
the update equation for the message $\nu_{u\to v}$ from vertex $u$ to its neighboring vertex $v$ in loopy BP can be approximated as
\begin{align*}
	\nu_{u\to v}(a) \propto \psi_{u}(a) e^{-h(a)} \left\{ \prod_{w \in N_u \setminus \{v\}} \sum_{b} c_{ab} \nu_{w\to u}(b) \right\}
	,
\end{align*}
where $h(a) \coloneqq \frac{1}{N}\sum_{w \in V} \sum_{b} c_{ab} \nu_{w}(b)$ aggregates the influence from non-adjacent vertices, and $\nu_{w}(b)$ denotes the marginal probability that a vertex $w$ belongs to cluster $b$.

In order to clarify the relationship between the CNBT matrix and BP, we introduce the following assumptions: (i) each vertex shares the same prior and that is uniform, $\psi_u(a) = n_a = \frac{1}{K}$, and (ii) the vertices in each cluster have approximately the same degree, $c \coloneqq \sum_{b} c_{ab}$. Defining $\nu_{u\to v}(a) = n_a + \delta_{u\to v}(a)$ and linearizing around $n_a$ leads
\begin{align*}
	\delta_{u\to v}(a) = \sum_{w\in N_u \setminus \{ v \}} \sum_{b} \frac{c_{ab}}{c} \delta_{w\to u}(b)
	.
\end{align*}
Let $\frac{c_{ab}}{c}$ be the $(a,b)$ component of a matrix $T$. A vector $\bm{\delta} \in \mathbb{C}^{2mK}$ and a function $\text{mat}:\mathbb{C}^{2mK} \to \mathbb{C}^{2m\times K}$ are defined as follows:
\begin{align*}
	\bm{\delta}             & \coloneqq \begin{pmatrix} \delta_{1}(1) & \delta_{2}(1) & \dots & \delta_{2m-1}(K) & \delta_{2m}(K) \end{pmatrix}
	,                                                                                                                                           \\
	\text{mat}(\bm{\delta}) & \coloneqq
	\begin{pmatrix}
		\delta_{1}(1)  & \dots  & \delta_{1}(K)  \\
		\vdots         & \ddots &                \\
		\delta_{2m}(1) &        & \delta_{2m}(K)
	\end{pmatrix}
	= \begin{pmatrix}
		  \bm{\delta}_{1} & \dots & \bm{\delta}_{2m}
	  \end{pmatrix}^{\top}
	.
\end{align*}
This leads to
\begin{align}
	\bm{\delta} = (T \otimes B) \bm{\delta}
	.\label{eq:delta_equation}
\end{align}
We note that Eq.~\eqref{eq:delta_equation} is equivalent to $\text{mat}(\bm{\delta}) = B \text{mat}(\bm{\delta}) T^{\top}$.

Here, we consider the following two types of matrices $T_1, T_2 \in \mathbb{R}^{K\times K}$ with $e \in (0.5,1), f, g \in \mathbb{R}_{+}, f > g$:
\begin{align*}
	(T_1)_{ab} & \coloneqq
	\begin{cases}
		e   & a+1 \equiv b \text{ mod } K
		\\
		1-e & b+1 \equiv a \text{ mod } K
		\\
		0.5 & \text{otherwise}
		,
	\end{cases}
	\quad      &
	(T_2)_{ab} & \coloneqq
	\begin{cases}
		f & a+1 \equiv b \text{ mod } K
		\\
		g & \text{otherwise}.
	\end{cases}
\end{align*}
These matrices model a scenario where directed edges are more likely to occur in a cyclic orientation, such as from cluster $1$ to $2$, $2$ to $3$, and so on.
The eigenvalues of $T_1, T_2$ satisfy the following:
\begin{lemma}\label{lemma:spectral_matrix_T}
	The matrix $T_1$ has a single real eigenvalue and all remaining eigenvalues are purely imaginary. Similarly, the matrix $T_2$ has a single real eigenvalue, with the others denoted as $(f-g) \exp \left( \frac{2\pi \sqrt{-1}}{K} k \right),\; k=1,\dots, K-1$.
\end{lemma}
Suppose that we specifically choose $\bm{\delta}_i$ to be one of the eigenvectors of $T$, with $\mu_i \in \mathbb{C}$ as the corresponding eigenvalue, then
\begin{align*}
	B \text{mat}(\bm{\delta}) T^{\top} = B \; \text{diag}\begin{pmatrix} \mu_1 & \dots & \mu_{2m} \end{pmatrix} \text{mat}(\bm{\delta})
	.
\end{align*}
In Eq.~\eqref{eq:delta_equation}, if we assume that the matrix $T$ is either $T_1$ or $T_2$,
their eigenvalues include real multiples of the imaginary unit $\sqrt{-1}$ or $\exp(\frac{2\pi \sqrt{-1}}{K} k), k=1,\dots, K-1$.
Consequently, depending on the choice of $\bm{\delta}_i$ and model parameters, $B\; \text{diag}(\mu_1\; \dots\; \mu_{2m})$ and $B_{\alpha} = B\Lambda$ exhibit a structural similarity. This leads a connection between BP and the eigenvalue problem of the CNBT matrix.

\subsubsection{Discussion}\label{section:discuss}

The analysis in the above section shows that a close relationship between the spectral properties of the CNBT matrix and BP on directed graphs.
However, this relationship has not been fully elucidated, and several aspects still require further investigation.

First, based on the above derivation, the exchanged messages in BP may become complex numbers. When considering directed graphs, the matrix $T$ is asymmetric, leading to complex eigenvectors. Consequently, $\bm{\delta}$ also becomes a complex vector. The interpretation of BP with such complex-valued messages requires further consideration.

Second, It remains unclear what conditions eigenvectors must satisfy to serve as a stable solution in BP.
We have only established that one possible form of the equations involves the CNBT matrix.
However, whether this actually corresponds to a stable BP solution must be investigated further.

Finally, an important concern is the extent to which the assumptions introduced in the previous section (namely, a uniform prior shared by all vertices and the degree homogeneity) can be generalized or relaxed.
Addressing these open questions is critical for advancing our understanding of the properties of complex-valued matrix representations.

\section{Conclusion}\label{section:conclusion}

We introduced the complex non-backtracking matrix corresponding to the Hermitian adjacency matrix for directed graphs.
We showed that our definition extends the Ihara's formula to the Hermitian matrices
and revealed analogous results for relationships involving in/out-vectors and properties of mixed walk counts, similar to the case of undirected graphs.
Moreover, we experimentally demonstrated that the eigenvectors of the proposed matrix representation are useful for clustering in directed graphs, similar to the case in sparse undirected graphs. We also discussed the connection to belief propagation.

To further explore our findings, several issues need to be addressed.
Firstly, as described in~\ref{section:discuss},
the relationship between the complex non-backtracking matrix and belief propagation has not yet been fully elucidated.
Addressing issues such as the treatment of complex-valued message in belief propagation,
its convergence criteria, and the analysis of cluster estimation thresholds via the study of phase transitions is expected to lead to the development of algorithms with improved performance.

Additionally, there is clear challenge in normalization for clustering algorithm.
The proposed algorithm employs standard normalization for fair comparison with other methods.
However, it has been observed that the data points of vertices calculated from the eigenvectors of the complex non-backtracking matrix, tend to concentrate around the origin, particularly for vertices that are more difficult to estimate a cluster.
This tendency has also been pointed out in~\citep{Newman2013-zc}, where a specific normalization unique to the non-backtracking matrix is proposed.
Investigating the interpretation of this specific normalization in belief propagation could offer a valuable direction for future research.

Furthermore, while it is empirically known that
the eigenvalue corresponding to the eigenvector with cluster structure has the largest real part
and an almost zero imaginary part, the underlying connection remains unclear.
Exploring this relationship could potentially lead to the development of metrics for quantifying the strength and stability of cluster structures.

\section*{Acknowledgment}
This work was supported by JSPS KAKENHI Grant Nos. JP22H03653 and 23H04483.

\appendix

%%%%%%%%%%%%%%%%%%%%%%%%%%%%%%%%%%%%%%%%%%%%%%%%%
\section{An example of each notation}\label{appendix:notations}

\begin{figure}[tb]
	\centering
	\begin{tikzpicture}
		\Vertex[label=$v_1$, x=0.0, y=1.5, style={fill=none}]{v1}
		\Vertex[label=$v_2$, x=1.5, y=1.5, style={fill=none}]{v2}
		\Vertex[label=$v_3$, x=0.0, y=0.0, style={fill=none}]{v3}
		\Vertex[label=$v_4$, x=1.5, y=0.0, style={fill=none}]{v4}

		\Edge[Direct, bend=10](v1)(v2)
		\Edge[Direct, bend=10](v2)(v1)
		\Edge[Direct, bend=-10](v1)(v3)
		\Edge[Direct, bend=10](v2)(v3)
		\Edge[Direct, bend=-10](v3)(v4)
		\Edge[Direct, bend=-10](v4)(v2)
	\end{tikzpicture}
	\caption{An example of the directed graph $G=(V, E)$. We set $R=4$.}
	\label{figure1:appendix:notations}
\end{figure}
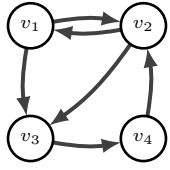

\begin{table}[tb]
	\caption{Examples of mixed walks \& cycles}
	\begin{tabular}{llcc}\hline
		mixed walk                                  & type                        & $\left| W \right|$ & $r(W)$
		\\\hline
		$W_1 = (v_2, v_1)$                          & NBT                         & $1$                & $0$
		\\\hline
		$W_2 = (v_2, v_3)$                          & NBT                         & $1$                & $1$
		\\\hline
		$W_3 = (v_4, v_3)$                          & NBT                         & $1$                & $3$
		\\\hline
		$W_4 = (v_1, v_2, v_3, v_4)$                & NBT                         & $3$                & $2$
		\\\hline
		$W_5 = (v_1, v_2, v_3, v_2)$                & backtracking                & $3$                & $0$
		\\\hline
		$C_1 = (v_1, v_2, v_3, v_1)$                & NBT, primitive              & $3$                & $0$
		\\\hline
		$C_2 = (v_1, v_2, v_1, v_2, v_1)$           & backtracking, not primitive & $4$                & $0$
		\\\hline
		$C_3 = (v_1, v_2, v_3, v_4, v_2, v_1)$      & NBT, tail, primitive        & $5$                & $3$
		\\\hline
		$C_4 = (v_2, v_4, v_3, v_2, v_4, v_3, v_2)$ & NBT, not primitive          & $6$                & $2$
		\\\hline
	\end{tabular}
	\label{table1:appendix:notations}
\end{table}

In this appendix, we provide a simple directed graph to demonstrate each concept. A concrete graph we use here is shown in Figure~\ref{figure1:appendix:notations}.
This graph $G=(V,E)$ is defined as $V = \left\{ v_1, v_2, v_3, v_4 \right\}$, $E = \left\{ e_{1,2}, e_{2,1}, e_{1,3}, e_{2,3}, e_{3,4}, e_{4,2} \right\}$
where $e_{i,j} \coloneqq \overrightarrow{v_i v_j}$. Then, we can index edges as
\begin{align*}
	\vec{E}
	 & = \left\{
	e_1,\; e_2,\; e_3,\; e_4,\; e_5,\;
	e_6,\; e_7,\; e_8,\; e_9,\; e_{10}
	\right\}
	,            \\
	 & = \left\{
	e_{1,2},\; e_{1,3},\; e_{2,3},\; e_{3,4},\; e_{4,2},\;
	e_{2,1},\; e_{3,1},\; e_{3,2},\; e_{4,3},\; e_{2,4}
	\right\}
	.
\end{align*}
Moreover, the relations
\begin{align*}
	\left\{
	v_1 \leftrightarrow_{G} v_2,\; v_2 \leftrightarrow_{G} v_1,\;
	v_1 \rightarrow_{G} v_3,\; v_3 \leftarrow_{G} v_1,\;
	v_2 \rightarrow_{G} v_3,\; v_3 \leftarrow_{G} v_2,\;
	v_2 \leftarrow_{G} v_4,\; v_4 \rightarrow_{G} v_2,\;
	v_3 \rightarrow_{G} v_4,\; v_4 \leftarrow_{G} v_3
	\right\}
\end{align*}
are satisfied.
When we set $R=4$, then $\alpha = e^{\frac{2\pi \sqrt{-1}}{4}} = \sqrt{-1}$.
Matrices $A_{\alpha}$ and $B_{\alpha}$ are shown as follows, respectively.
\begin{align*}
	A_{\alpha} & = \begin{pmatrix}
		               0            & 1            & \alpha       & 0
		               \\
		               1            & 0            & \alpha       & \bar{\alpha}
		               \\
		               \bar{\alpha} & \bar{\alpha} & 0            & \alpha
		               \\
		               0            & \alpha       & \bar{\alpha} & 0
	               \end{pmatrix}
	,          & \quad
	B_{\alpha} & = \begin{pmatrix}
		               0 & 0      & \alpha & 0      & 0      & 0 & 0            & 0            & 0            & \bar{\alpha}
		               \\
		               0 & 0      & 0      & \alpha & 0      & 0 & 0            & \bar{\alpha} & 0            & 0
		               \\
		               0 & 0      & 0      & \alpha & 0      & 0 & \bar{\alpha} & 0            & 0            & 0
		               \\
		               0 & 0      & 0      & 0      & \alpha & 0 & 0            & 0            & 0            & 0
		               \\
		               0 & 0      & \alpha & 0      & 0      & 1 & 0            & 0            & 0            & 0
		               \\
		               0 & \alpha & 0      & 0      & 0      & 0 & 0            & 0            & 0            & 0
		               \\
		               1 & 0      & 0      & 0      & 0      & 0 & 0            & 0            & 0            & 0
		               \\
		               0 & 0      & 0      & 0      & 0      & 1 & 0            & 0            & 0            & \bar{\alpha}
		               \\
		               0 & 0      & 0      & 0      & 0      & 0 & \bar{\alpha} & \bar{\alpha} & 0            & 0
		               \\
		               0 & 0      & 0      & 0      & 0      & 0 & 0            & 0            & \bar{\alpha} & 0
	               \end{pmatrix}
	.
\end{align*}

When we consider the cycle $C_1$ in Table~\ref{table1:appendix:notations},
the $2$ times repetition of the cycle is $( C_1 )^2 = (v_1, v_2, v_3, v_1, v_2, v_3, v_1)$.
In addition, $C_1$ is equivalent with a cycle $(v_2, v_3, v_1, v_2)$.

The NBT walks $W_1,\; W_2,\; W_3$ and $W_4$ are counted in $( P_{(1,0)} )_{v_2 v_1}$, $( P_{(1,1)} )_{v_2 v_3}$, $( P_{(1,3)} )_{v_4 v_3}$ and $( P_{(3,2)} )_{v_1 v_4}$, respectively.
The NBT cycles $C_1, C_3$ and $C_4$ can be counted within multiple elements of $Q_{(k,r)}$.
For instance, $C_4$ is included in elements $( Q_{(6,2)} )_{e_8 e_8}$ and $( Q_{(6,2)} )_{e_1 e_8}$, among others.
However, when it comes to the diagonal components of $Q_{(k,r)}$, a cycle is not redundantly included in multiple diagonal elements.
This means that $C_4$ is only included in $( Q_{(6,2)} )_{e_8 e_8}$ among the diagonal elements of $Q_{(6,2)}$.

%%%%%%%%%%%%%%%%%%%%%%%%%%%%%%%%%%%%%%%%%%%%%%%%%%%%
\section{Derivation of Eq.~\eqref{eq:delta_equation}}\label{appendix:derivation_delta}

Belief propagation~(BP) is a message-passing algorithm for calculating a probability distribution on graphical models.
In this paper, we investigate the connection between the CNBT matrix and cluster estimation via BP.
This appendix provides a detailed derivation of Eq.~\eqref{eq:delta_equation}.

For the standard BP notation, we define $\psi_{uv}(a,b) \coloneqq P_{a b}^{A_{uv}} (1-P_{a b})^{1-A_{uv}}$.
Using this notation, the update equation of a message $\nu_{u\to v}$ and the marginal distribution $\nu_{u}$ can be expressed as
\begin{align*}
	\nu_{u\to v}(t_u) & \propto \psi_u(t_u) \prod_{w\in V\setminus \{u, v\}} \sum_{t_w} \psi_{uw}(t_u, t_w) \nu_{w\to u}(t_w)
	,                                                                                                                         \\
	\nu_u(t_u)        & \propto \psi_u(t_u) \prod_{w\in V\setminus \{ u \}} \sum_{t_w} \psi_{uw}(t_u, t_w) \nu_{w\to u}(t_w)
	.
\end{align*}
With the normalization condition $\sum_{b} \nu_{w\to u}(b) = 1$ and $P_{ab} = \frac{c_{ab}}{N}$, the update equation can be rewritten by
\begin{align*}
	\nu_{u\to v}(t_u) \propto \psi_u(t_u) \left\{ \prod_{w\in N_u \setminus \{ v \}} \sum_{t_w} c_{t_u t_w} \nu_{w\to u}(t_w) \right\} \left\{ \prod_{\substack{w\in V\setminus N_u \\ w\neq u,v}} \left( 1-\frac{1}{N} \sum_{t_w} c_{t_u t_w} \nu_{w\to u}(t_w) \right) \right\}
	.
\end{align*}
Considering whether there exists an edge between vertex $u$ and vertex $v$, we apply different approximations accordingly.

When there does not exist an edge from vertex $u$ to vertex $v$,
\begin{align*}
	\nu_{u\to v}(t_u) \propto \frac{\nu_u(t_u)}{1 - \frac{1}{N} \sum_{t_v} c_{t_u t_v} \nu_{v\to u}(t_v)}
	\xrightarrow{N\to\infty} \nu_u(t_u)
	.
\end{align*}
Let us consider the case where an edge exists. Since the graph is sparse, the degree of each vertex $u$ is $O(1)$, which implies that $\frac{1}{N} \sum_{w\in N_u \cup \{ u \}} \sum_{t_w} c_{t_u t_w} \nu_w(t_w) \xrightarrow{N\to\infty} 0$ holds.
Furthermore, by analogy with $\lim_{n\to\infty} (1+\frac{x}{n})^n = e^x$, the following approximation holds.
\begin{align}
	\nu_{u\to v}(t_u) & \propto \psi_u(t_u) e^{-h(t_u)} \prod_{w\in N_u \setminus \{ v \}} \sum_{t_w} c_{t_u t_w} \nu_{w\to u}(t_w)
	,                                                                                                                               \label{eq:adjacent_message} \\
	\text{where}      & \quad h(a) \coloneqq \frac{1}{N} \sum_{w\in V} \sum_{b} c_{ab} \nu_{w}(b)
	.
\end{align}
The term $h(a)$ is said to be an auxiliary external field and summarizes the influence from non-adjacent vertices.

By using $\nu_{u\to v}(a) = n_a + \delta_{u\to v}(a)$ and $\log(C + x) \approx \log C + \frac{x}{C}$, the term $\frac{\nu_{u\to v}(a)}{\nu_{u\to v}(b)}, a\neq b$ can derive
\begin{align*}
	\frac{\delta_{u\to v}(a)}{n_a} - \frac{\delta_{u\to v}(b)}{n_b} =
	- (h(a) - h(b)) + \sum_{w\in N_u \setminus \{ v \}} \sum_{d} \frac{1}{c}(c_{a d} - c_{b d}) \delta_{w\to u}(d)
	.
\end{align*}
Since this equation holds for any $b\neq a$, summing over $b\neq a$ yields Eq.~\eqref{eq:delta_equation}.

\section{Proof of Lemma~\ref{lemma:spectral_matrix_T}}

Since the matrices $T_1, T_2$ are circulant, let $p_0 = \frac{1}{2},\; p_1 = e,\; p_2 = \dots = p_{K-2} = \frac{1}{2},\; p_{K-1} = 1-e$ and $q_0 = g,\; q_1 = f,\; q_2 = \dots = q_{K-1} = g$, then $(T_1)_{ab} = p_{(b-a) \text{ mod } K}$ and $(T_2)_{ab} = q_{(b-a) \text{ mod } K}$ holds.
Furthermore, we denote the eigenvalues of $T_1, T_2$ by $\lambda_k, \mu_k, k=0,\dots, K-1$ respectively. A property of circulant matrices~\citep{Gray2005-ne} leads
\begin{align*}
	\lambda_k & = \sum_{l=0}^{K-1} p_l \exp\left( \frac{2\pi \sqrt{-1}}{K} kl \right)
	,         & \quad
	\mu_k     & = \sum_{l=0}^{K-1} q_l \exp\left( \frac{2\pi \sqrt{-1}}{K} kl \right)
	.
\end{align*}

When $k=0$, $\lambda_0 = \frac{K}{2}$ and $\mu_0 = f + (K-1)g$ hold. For the case where $k\neq 0$, $\lambda_k$ and $\mu_k$ can be calculated as follows:
\begin{align*}
	\lambda_k & = \sqrt{-1} (2e - 1) \sin\left( \frac{2\pi}{K} k \right)
	,         & \quad
	\mu_k     & = (f-g) \exp\left( \frac{2\pi \sqrt{-1}}{K} k \right)
	.
\end{align*}
\qed

\bibliography{main}
\end{document}